\documentclass[a4paper]{amsart}

\usepackage[
  margin=30mm,
  marginparwidth=25mm,     
  marginparsep=2mm,       
  bottom=25mm,
  ]{geometry}

\usepackage[latin1]{inputenc}
\usepackage[T1]{fontenc}
\usepackage[spanish,english]{babel}
\usepackage{amsmath,amssymb,amsthm}
\usepackage{latexsym}
\usepackage{delarray}
\usepackage{bbm}
\usepackage{hyperref}
\usepackage[pdftex,usenames,dvipsnames]{color}
\usepackage{wasysym}
\usepackage{datetime}
\usepackage{mathrsfs}

\usepackage{tikz}

\usepackage{amsmath}

\newtheorem{Th}{Theorem}[section]
\newtheorem{Prop}[Th]{Proposition}
\newtheorem{Lem}[Th]{Lemma}
\newtheorem{Def}[Th]{Definition}

\numberwithin{equation}{section}

\newcommand{\N}{\mathbb{N}}
\newcommand{\R}{\mathbb{R}}
\newcommand{\Rn}{\mathbb{R}^{n+1}}
\newcommand{\Z}{\mathbb{Z}}

\newcommand{\G}{\Gamma}

\newcommand{\HH}{\mathcal{H}}
\newcommand{\LL}{\mathcal{L}}

\newcommand{\M}{\mathcal{M}}
\newcommand{\Q}{\mathcal{Q}}

\newcommand{\eps}{\varepsilon}

\DeclareMathOperator{\supp}{supp}
\newcommand{\dv}{\operatorname{div}}

\def\mean#1{\mathchoice%
          {\mathop{\kern 0.2em\vrule width 0.6em height 0.69678ex depth -0.58065ex
                  \kern -0.8em \intop}\nolimits_{\kern -0.4em#1}}%
          {\mathop{\kern 0.1em\vrule width 0.5em height 0.69678ex depth -0.60387ex
                  \kern -0.6em \intop}\nolimits_{#1}}%
          {\mathop{\kern 0.1em\vrule width 0.5em height 0.69678ex
              depth -0.60387ex
                  \kern -0.6em \intop}\nolimits_{#1}}%
          {\mathop{\kern 0.1em\vrule width 0.5em height 0.69678ex depth -0.60387ex
                  \kern -0.6em \intop}\nolimits_{#1}}}

\title[Homogenization of a parabolic Dirichlet problem by a method of Dahlberg]{Homogenization of a parabolic Dirichlet problem by a method of Dahlberg}

\author[A. J. Castro]{Alejandro J. Castro}
\author[M. Str\"omqvist]{Martin Str\"omqvist}

\address{\newline
        Alejandro J. Castro \newline
        Department of  Mathematics, Nazarbayev University, \newline
		010000 Astana, Kazakhstan}
\email{alejandro.castilla@nu.edu.kz}

\address{\newline
        Martin Str\"omqvist \newline
        Department of  Mathematics, Uppsala University, \newline
        S-751 06 Uppsala, Sweden}
\email{martin.stromqvist@math.uu.se}


\begin{document}

\footnotetext{Last modification: \today.}

\begin{abstract}
Consider the linear parabolic operator in divergence form 
$$\HH u =\partial_t u(X,t)-\text{div}(A(X)\nabla u(X,t)).$$ 
We employ a method of Dahlberg to show that the Dirichlet problem for $\HH$ in the upper half plane is well-posed for boundary data in $L^p$, for any elliptic matrix of coefficients $A$ which is periodic and satisfies a Dini-type condition. This result allows us to treat a homogenization problem for the equation $\partial_t u_\eps(X,t)-\text{div}(A(X/\eps)\nabla u_\eps(X,t))$ in Lipschitz domains with $L^p$-boundary data.
\end{abstract}

\maketitle


\section{Introduction, notation and main results}


In this paper we are interested in the well-posedness of low regularity Dirichlet problems associated with the divergence type parabolic operator
$$\HH u=\partial_t u - \dv \Big( A(X,t) \cdot \nabla u \Big),$$
for a certain periodic matrix of coefficients $A$. That is, we would like to guarantee existence and uniqueness of solutions and continuous dependence on the boundary data, under minimal regularity assumptions on the coefficients and on the domain. For the upper half space $$\{(x,t,\lambda):x\in\R^n,\;t\in \R,\;\lambda>0\},$$ 
we prove that the $L^p$ Dirichlet problem is well-posed if $A$ is periodic in the $\lambda$-direction. This extends previous results for the upper half space, where it is assumed that $A$ is either independent of $\lambda$, or that $A$ is a perturbation of a matrix that is independent of $\lambda$. The theory developed for the upper half space allows us to study homogenization problems in bounded, time-independent Lipschitz domains. 

We start by briefly putting these problems into context, mentioning just a few papers that precede this work. 
For the ordinary heat equation, in which case the matrix $A$ is simply the identity matrix, Fabes and Rivi\`ere (\cite{FR}) established the solvability in $C^1$--cylinders. Later, Fabes and Salsa (\cite{FS}) and Brown (\cite{Br}) extended the result to Lipschitz cylinders. For more involved time-varying domains, the situation has been analyzed by Lewis and Murray (\cite{LewMur}) and Hofmann and Lewis (\cite{HofLew}).
The next step was to allow non-constant coefficients. Mitrea
(\cite{Mitr}) studied the situation of $A \in C^\infty$; Castro, Rodr\'iguez-L\'opez and Staubach (\cite{CRS}) considered H\"older matrices and Nystr\"om (\cite{N2}) the case of complex elliptic matrices, but independent of one of the spatial variables.

In all previous contexts, the matrices were time-independent. Allowing time-dependence is a very challenging problem, which has been understood very recently by Auscher, Moritz and Nystr\"om (\cite{AMN}), following a first order approach. They consider elliptic matrices depending on time and all spatial variables, which are certain perturbations of matrices independent of one single spatial direction (see \cite[Section 2.15]{AMN} for precise definitions).

It is also worth noting that in almost all the aforementioned
papers, the analysis was carried out via the so called method of layer potentials, that we will not follow this time here. We consider the parabolic Dirichlet problem in Lipschitz cylinders for merely elliptic coefficients, depending on all spatial variables. However, we need to assume periodicity in one direction and a Dini-type condition in the same variable, as made precise below. 

We show that if the coefficient matrix $A$ is time-independent and periodic with period $1$ in the spatial direction of the normal of the boundary, then the Dirichlet problem is solvable. Moreover, the estimates that we obtain for the solution are independent of the period of $A$. For periodic matrices $A(X)$ and $\eps>0$, we can then obtain estimates that are uniform in $\eps$ for the solution $u_\eps$ to the Dirichlet problem with coefficient matrix $A(x/\eps)$ with period $\eps$. In particular, we prove that, as $\eps\to0$, $u_\eps$ converges to a limit function $\bar u$ that solves the Dirichlet problem with a constant coefficient matrix $\bar A$. A limit process of this type is called homogenization. For elliptic operators, these estimates were obtained by Kenig and Shen in \cite{KS}. In \cite{KS}, the authors have two independent ways of proving the estimates. The first is an approximation argument that relies on certain integral identities, the second is through a potential theoretic method due to Dahlberg. For parabolic problems these integral estimates are not available and we rely instead on a parabolic version 
of the theorem by Dahlberg. 

Let $\HH$ denote the parabolic operator
$$\HH u := (\partial_t + \LL)u,$$
where
\begin{equation*}\label{operator}
\LL u := - \dv \Big( A(X,t) \cdot \nabla u \Big)
	 = - \sum_{i,j=1}^{n+1} \partial_{x_i} (A_{i,j}(X,t) \partial_{x_j}u),
\end{equation*}
is defined in $\R^{n+2}=\{ (X,t)=(x_1, \dots, x_{n+1},t) \in \R^{n+1} \times \R\}$, $n \geq 1$;
and $A=\{A_{i,j}(X,t)\}_{i,j=1}^{n+1}$ is an $(n+1)\times(n+1)$ real and symmetric matrix which satisfies:
\begin{itemize}

\item for certain $1 \leq \Lambda < \infty$, the \textit{uniform ellipticity} condition
\begin{equation}\label{eq:Aellip}
  \Lambda^{-1} |\xi|^2 \leq \sum_{i,j=1}^{n+1} A_{i,j}(X,t)\xi_i \xi_j \leq \Lambda |\xi|^2, \quad \xi \in \R^{n+1};
\end{equation}

\item independence of the time variable $t$, 
\begin{equation}\label{eq:Atindep}
	A(X,t)=A(X);
\end{equation}

\item periodicity in the $x_{n+1}$ variable
\begin{equation}\label{eq:Aperiod}
  A(x,x_{n+1} + 1)=A(x,x_{n+1}), \quad x \in \R^n, \ x_{n+1} \in \R;
\end{equation}

\item a Dini-type condition in the $x_{n+1}$ variable
\begin{equation}\label{eq:Dini}
  \int_0^1 \frac{\theta(\rho)^2}{\rho} d\rho < \infty,
\end{equation}
where $\theta(\rho):= \{|A(x,\lambda_1) - A(x,\lambda_2)| \ : \ x\in \R^n, \ |\lambda_1 - \lambda_2| \leq \rho \}$.
\end{itemize}

\quad\\
In virtue of the hypothesis \eqref{eq:Atindep} and \eqref{eq:Dini}, the $x_{n+1}$ direction is of special interest. 
Along this paper we call $\lambda:=x_{n+1}$. Accordingly, 
$\nabla
  :=(\nabla_{||},\partial_\lambda)
  :=(\partial_{x_1},\dots,\partial_{x_n},\partial_\lambda)$.
  Depending on the situation, we refer to a point in $\R^{n+2}$ either as $(X,t)$, $X=(x,\lambda)$, or $(x,t,\lambda)$, with an obvious abuse of notation. The latter is convenient when we consider the Dirichlet problem in the upper half space, where $(x,t)$ denotes a point on the boundary.
  

Our theorems are formulated in time-independent Lipschitz domains. By $D$ we denote the domain 
\begin{equation}\label{unbounded}
D := \{(x,t,\lambda)\in \R^n\times\R\times\R:\lambda>\phi(x)\}, 
\end{equation}
which is an unbounded cylinder in time, whose spatial base is the region above the Lipschitz graph $\phi$, i.e., $\phi$ satisfies
$$|\phi(x)-\phi(y)|\le m|x-y|, \quad x,y \in \R^n,$$
for certain $m>0$.
The (lateral) boundary of $D$ is given by
$$\partial D := \{(x,t,\phi(x))  : x \in \R^n, \ t \in \R \}.$$
We shall also consider bounded Lipschitz cylinders 
\begin{equation}\label{cyl}
\Omega_T:=\Omega\times(0,T), \text{ where }\Omega \text{ is a bounded Lipschitz domain in }\R^{n+1}.
\end{equation}
It will be assumed that $\Omega$ is a $(m,r_0)$ domain in the following sense: 
\[\left\{\begin{array}{l}
\text{For any }X_0\in \partial\Omega, \text{ there exists a Lipschitz  continuous function }\phi\text{ such that,}\\
\text{after a rotation of the coordinates, one has }X_0=(x_0,\lambda_0)\text{ and}\\
\{(x,\lambda):|x-x_0|<r_0,\;|\lambda-\lambda_0|<mr_0\}\cap\Omega = 
\{(x,\lambda):|x-x_0|<r_0,\;\phi(x)<\lambda<mr_0\}. 
\end{array}\right.\]
Thus, introducing 
\[U(x_0,t_0,\lambda_0) = \{(x,t,\lambda):|x-x_0|<r_0,\;|t-t_0|<r_0^2,\;|\lambda-\lambda_0|<mr_0\},\]
one has 
\begin{equation}\label{localbdry}
U(x_0,t_0,\lambda_0)\cap\Omega_T = \{(x,t,\lambda)\in\R^{n+2}:\phi(x)<\lambda\}\cap U(x_0,t_0,\lambda_0)\cap \{0<t<T\}.
\end{equation}
The lateral boundary of $\Omega_T$ is denoted by $\partial_L\Omega_T:=\partial\Omega\times(0,T)$ and the parabolic boundary is given by $\partial_P\Omega_T:=\overline\Omega\times\{0\}$. Note that for $D$ as in \eqref{unbounded}, $\partial_L D=\partial D$ and $\partial_P D=\emptyset$. On $\partial_L\Omega_T$ and $\partial D$ we define $L^p$ spaces with respect to the measure 
\begin{equation}\label{measure}
d\sigma(X,t)=d\sigma(X)dt,
\end{equation}
where $\sigma$ is the surface measure on $\partial\Omega$ and $\{(x,\phi(x)):x\in\R^n\}$, respectively. 



We shall need to introduce some more notation that will be needed to state our main results.
For $(X,t) \in \R^{n+1} \times \R$, we define its \textit{parabolic norm} $||(X,t)||$
as the unique positive solution $\rho$ of the
equation
$$\frac{t^2}{\rho^4} + \sum_{i=1}^{n+1} \frac{x_i^2}{\rho^2}=1.$$
It satisfies that $||(\gamma X,\gamma^2t)||=\gamma||(X,t)||$, $\gamma>0$. 
If $(x,t)\in\R^n\times\R$, we let $||(x,t)||=||(x,0,t)||$. 
We define the parabolic distance from $(X,t)\in \R^{n+2}$ to $(Y,s)\in\R^{n+2}$ by 
$d(X,t,Y,s)=||(X-Y,t-s)||$. 

Given $(x_0,t_0) \in \Rn$ and $\eta>0$, we define the \textit{cone}
$$\G^\eta(x_0,t_0)
  := \{(x,t,\lambda) \in \R_+^{n+2} : ||(x-x_0,t-t_0)|| < \eta \lambda\},$$
 and the \textit{standard parabolic cube} centered at $(x,t) \in \Rn$ with side length $\ell(Q)=r>0$ by
 $$Q:=Q_r(x,t)
   :=\{(y,s) \in \R^{n+1} : |y_i-x_i|<r, \ |t-s|<r^2\}.$$
Similarly, we consider parabolic cubes $\widetilde{Q}$ in $\R^{n+2}$ centered at $(X,t)$ as follows,
 $$\widetilde{Q}:=\widetilde{Q}_r(X,t)
   :=\{(Y,s) \in \R^{n+2} : |Y_i-X_i|<r, \ |t-s|<r^2\}.$$
It will also be useful to introduce the set
\[
T_r(x,t) := Q_{r}(x,t)\times(0,r).
\]

For any function $u$ defined in $\R^{n+2}_+:=\{(x,t,\lambda)\in \R^{n+2}:\lambda>0\},$  we consider the following \textit{non-tangential maximal operator}
\begin{equation*}\label{eq:N*}
N^{\eta}(u)(x_0,t_0)
  := \sup_{(x,t,\lambda) \in \G^\eta(x_0,t_0)} |u(x,t,\lambda)|. 
\end{equation*}
If $f(X,t)$ is defined on $\partial D$ and $(X_0,t_0)\in\partial D$, we say that $u(X_0,t_0)=f(X_0,t_0)$ non-tangentially (n.t.) if 
\[
\lim_{\substack{(Y,s)\in \Gamma^\eta(X_0,t_0)\\(Y,s)\to(X_0,t_0)}}u(Y,s)=f(X_0,t_0), 
\]
where $\eta$ is chosen such that $\partial D\cap\Gamma^\eta(X_0,t_0)=\{(X_0,t_0)\}$, i.e\ $\eta>M$. Having made such a choice of $\eta$ we simply denote $N(u) = N^\eta(u)$. 

In all our estimates $C$ denotes a constant that depends only upon the dimension $n$, the ellipticity constant $\Lambda$ and possibly $m,r_0$. 



\begin{Th}\label{per1}
Suppose that $A$ is a real and symmetric matrix satisfying \eqref{eq:Aellip} -- \eqref{eq:Dini}
and $D$ is an unbounded Lipschitz domain defined as in \eqref{unbounded}.
Then, for certain $0<\delta<1$ and any $f\in L^p(\partial D)$, $2-\delta<p<\infty$, there exists a unique solution to the Dirichlet problem 
\begin{equation*}
\left\{\begin{aligned}
\HH u&=0\quad\text{in }D,\\
u&=f\quad \text{n.t. on }\partial D,\\
\end{aligned}\right.
\end{equation*}
verifying 
\[
\|N(u)\|_{L^p(\partial D)}
	 \le C \|f\|_{L^p(\partial D)}. 
\]
\end{Th}

With Theorem \ref{per1} in place we are able to analyze a homogenization problem that we now describe. In addition to \eqref{eq:Aellip} and \eqref{eq:Atindep} we assume that 
\begin{equation}\label{perA}
A(X+Z)=A(X),\quad \text{for all }Z\in \Z^{n+1},
\end{equation}
and
\begin{equation}\label{DiniA}
  \int_0^1 \frac{\Theta(\rho)^2}{\rho} d\rho < \infty,
\end{equation}
where $\Theta(\rho):= \{|A(X) - A(Y)| \ : \ X,Y\in \R^{n+1}, \ |X-Y| \leq \rho \}$.
That is, $A$ is periodic with respect to the lattice $\Z^{n+1}$ and satisfies a Dini condition in all variables. 

For each $\eps>0$, consider the operator $\LL_\eps$  given by 
\[
\LL_\eps u :=-\text{div}(A_\eps(X)\nabla u),\quad A_\eps(X) := A\left(\frac{X}{\eps}\right).
\]
We also need to introduce $\bar \LL$, 
\[
\bar \LL u := -\text{div}(\bar A\nabla u),
\]
where the matrix $\bar A$ is determined by
\begin{equation*}\label{barA}
\bar{A}^t\alpha := \int_{(0,1)^n}A^t\nabla w_\alpha dy,
\quad \alpha\in \R^{n+1},
\end{equation*}
and the auxiliary function $w_\alpha$ solves the problem
\begin{equation*}
\left\{\begin{array}{l}
-\text{div}\left(A^t\nabla w_\alpha\right)=0 \text{ in }(0,1)^{n+1}, 
\vspace{0.25cm} \\
w_\alpha-\alpha y\text{ is }1-\text{periodic}, \vspace{0.25cm} \\
\displaystyle \int_{(0,1)^{n+1}}(w_\alpha-\alpha y)dy=0.
\end{array}\right.
\end{equation*}

Now we can state our homogenization result.
\begin{Th}\label{th:homogenization} 
Suppose that $A$ is a real and symmetric matrix satisfying \eqref{eq:Aellip}, \eqref{eq:Atindep}, \eqref{perA} and \eqref{DiniA}.  Let $\Omega_T$ be as in \eqref{cyl}. Then for any $\eps>0$ and $f\in L^p(\partial_L\Omega_T)$, $2-\delta<p<\infty$, 
there exists a unique solution $u_\eps$ to the Dirichlet problem 
\begin{equation}\label{eqeps}
\left\{\begin{aligned}
 \partial_t u_\eps + \LL_\eps u_\eps & = 0\quad\text{in }\Omega_T,\\
u_\eps&=f\quad \text{n.t. on }\partial_L \Omega_T,\\
u_\eps(X,0)&=0 \quad \text{in }\overline{\Omega},\\
\end{aligned}\right.
\end{equation}
satisfying 
\begin{equation}\label{nteps}
\|N(u_\eps)\|_{L^p(\partial_L \Omega_T)}\le C\|f\|_{L^p(\partial_L \Omega_T)}. 
\end{equation}
Moreover, as $\eps\to0$, $u_\eps$ converges locally uniformly in $\Omega_T$ to $\bar u$, which is the unique solution to 
\begin{equation}\label{homogeneous}
\left\{\begin{aligned}
\partial_t \bar u + \bar\LL \bar u &= 0\quad\text{in }\Omega_T,\\
\bar u &= f\quad \text{n.t. on }\partial_L\Omega_T, \\
\bar u(X,0)& = 0 \quad \text{in }\overline{\Omega},\\
\end{aligned}\right.
\end{equation}
with
\[
\|N(\bar u)\|_{L^p(\partial_L \Omega_T)}\le C\|f\|_{L^p(\partial_L \Omega_T)}. 
\]
\end{Th}

 In the elliptic case, Theorem \ref{per1} and the first part of Theorem \ref{th:homogenization} (\eqref{eqeps} and \eqref{nteps}) was proved by Kenig and Shen in \cite{KS}. 
In \cite{KS} the authors also treat the Neumann and regularity problems. The theory for the Neumann and regularity problems is based on the use of integral identities to estimate certain nontangential maximal functions. These integral identities are not available in the parabolic case and thus homogenization of Neumann and regularity problems remain an interesting and challenging open problem. 

The main tools in our analysis are Harnack inequalities and the estimation of Green's function in terms of $\LL$-caloric measure and vice versa, see Section \ref{lemmas}. The main difficulty in the parabolic setting is the time-lag that is present in these estimates. Our requiring that the matrix $A$ is time independent and symmetric leads to spatial symmetry and time-invariance of Green's function, see \eqref{rearrange}. This becomes a key point in the proof of the parabolic version of Dahlbergs theorem in Section \ref{Dahlberg}

\section{The Dirichlet Problem}\label{Sec:2}

We now turn to the proof of Theorem \ref{per1}. Since $D$ is globally defined by a Lipschitz graph, the situation of the proof may be reduced to the upper half space in a standard way, see for example \cite[p. 905]{KS}. 
Thus, the goal of this section is to solve the Dirichlet problem for the operator $\HH$ in the upper half space $\R^{n+2}_+$ with given boundary data on $\partial \R^{n+2}_+\equiv \R^{n+1}$. \\

 
\begin{Def}\label{solvability}
We say that the Dirichlet problem for $\HH u =0$ in $\R^{n+2}_+$ is \textit{solvable} in $L^p$ if there exists $0<\delta<1$ such that for every $2-\delta<p<\infty$ and every $f \in C_c(\R^{n+1})$, the solution to the Dirichlet problem
\begin{equation}\label{classical}
\left\{\begin{aligned}
\HH u&=0\quad\text{in } \R^{n+2}_+,\\
u&=f\quad \text{n.t. on } \R^{n+1},\\
\end{aligned}\right.
\end{equation}
verifies
$$ \|N(u)\|_{L^p(\R^{n+1})}
	\le C \|f\|_{L^p(\R^{n+1})}.$$  
\end{Def}
It can be shown that \eqref{classical} has a unique solution by analyzing, for any $k=1,2,\ldots$, the problems
\begin{equation*}
\left\{\begin{aligned}
&\HH u^k=0\quad\text{in } T_k(0,0),\\
&u^k=f\quad \text{n.t. on } Q_k(0,0),\quad u^k=0\text{ on }\partial_L T_k(0,0)\setminus Q_k(0,0),\\
&u^k(X,-k^2)=0\text{ on }\overline{T_k(0,0)}\cap\{t=-k^2\},
\end{aligned}\right.
\end{equation*}
and define $u:=\lim_{k\to\infty}u^k$ which will solve \eqref{classical}. 
This allows us to define the $\LL$-caloric measure $\omega:=\omega^{Z,\tau}$ on $\R^{n+1}$, which satisfies 
\[u(Z,\tau)=\int_{\R^{n+1}}f(x,t)d\omega(x,t),\]
where $u$ is the solution to \eqref{classical}. 
If $U$ is an open subset of $\R^{n+1}$, we say that $u$ is $\LL$-caloric in $U$ if $\HH u = \partial_tu+\LL u=0$ in $U$. If $-\partial_t u+\LL u = 0$ in $U$, we say that $u$ is adjoint $\LL$-caloric in $U$. 
The caloric measure is a doubling measure, i.e.\ 
\begin{equation}\label{doubling}
\omega(Q_{2r}(x_0,t_0))\le C\omega(Q_r(x_0,t_0)), 
\end{equation}
see \cite{FGS} for a proof.
Assuming that $d\omega$ and $dxdt$ are mutually absolutely continuous, we define the kernel $K(Z,\tau;x,t)$ with respect to the point $(Z,\tau)\in \R^{n+2}_+$ by 
\begin{equation}\label{eq:kernelK}
K(Z,\tau;x,t):= \lim_{r\to0}\frac{\omega(Q_r(x,t))}{|Q_r(x,t)|}.
\end{equation}
The solution to \eqref{classical} may thus be represented as 
\[u(Z,\tau)=\int_{\R^{n+1}}K(Z,\tau;x,t)f(x,t)dxdt.\]

We recall that the solvability in $L^2$ of the Dirichlet problem in $\R^{n+2}_+$ (in the sense of Definition \ref{solvability}) is equivalent to the reverse H\"older inequality for the kernel $K$ (see Lemma \ref{RHsolv} below):
\begin{equation}\label{eq:revhold}
\left(\frac{1}{|Q_r(x,t)|}\int_{Q_r(x,t)}|K(Z,\tau;y,s)|^2dyds\right)^{1/2}\le \frac{C}{|Q_r(x,t)|}\int_{Q_r(x,t)}|K(Z,\tau;y,s)|dyds, 
\end{equation}
for all $(x,t)\in \R^{n+1}$ and all $(Z,\tau)\in \R^{n+2}_+$ for which $|(x,0)-Z|^2\le|t-\tau|$ and 
$\tau-t\ge 4 r^2$. 
The reverse H\"older inequality is self improving in the sense that if \eqref{eq:revhold} holds, then there exists $\alpha>2$ such that 
\begin{equation}\label{eq:revholdalpha}
\left(\frac{1}{|Q_r(x,t)|}\int_{Q_r(x,t)}|K(Z,\tau;y,s)|^\alpha dyds\right)^{1/\alpha}\le \frac{C}{|Q_r(x,t)|}\int_{Q_r(x,t)}|K(Z,\tau;y,s)|dyds.
\end{equation}
This is a consequence of Gehring's Lemma (\cite[Lemma 3]{Geh}), adapted to parabolic cubes. 
In turn, the reverse H\"older inequality is equivalent to the following condition
(see Proposition \ref{Prop:equiv2.3-2.4} below):
\begin{equation}\label{eq:loc}
\int_{Q_r(x_0,t_0)}\limsup_{\lambda\to0}\left|\frac{u(x,t,\lambda)}{\lambda}\right|^2dxdt\le \frac{C}{r^3}\int_{T_{2r}(x_0,t_0)}|u(x,t,\lambda)|^2dxdtd\lambda, \quad r>0,
\end{equation}
provided that $\HH  u =(\partial_t+\LL )u= 0$ in $T_{4r}(x_0,t_0)$ and $u(x,t,0)=0$ on $Q_{4r}(x_0,t_0)$.
Shortly, we call \eqref{eq:loc} a \textit{local solvability} condition when \eqref{eq:loc} holds for $0<r \leq 1$. 

If \eqref{classical} holds for $\HH^*= -\partial_t+\LL$ instead of $\HH = \partial_t+\LL$, we say that $u$ solves the adjoint Dirichlet problem. Analogously, we define the adjoint $\LL$-caloric measure $\omega^*$ and the adjoint kernel $K^*(Z,\tau;y,s)$. It is easy to see that the adjoint Dirichlet problem is solvable if and only if the Dirichlet problem for $\HH$ is solvable by considering the change of variables $t\mapsto -t$. This leads to analogous equivalent solvability conditions for the adjoint Dirichlet problem. For example, \eqref{eq:loc} holds for caloric functions if and only if it holds for adjoint caloric functions. 

Our first step in the proof of Theorem \ref{per1} is to establish \eqref{eq:loc} for $0<r<1$. This is achieved by localizing the operator and using the perturbation theory developed in \cite{N1}. Then we utilize an ingenious technique developed by Dahlberg to show that the periodicity of $A$ implies that \eqref{eq:loc} also holds for all $r>1$, see Theorem \ref{Prop:2.4forrgeq1} below. \\

For Lipschitz cylinders $\Omega_T=\Omega\times(0,T)$, we say that the $L^p$ Dirichlet problem is solvable in $\Omega_T$ if there exists $0<\delta<1$ such that for every $2-\delta<p<\infty$ and for every $f\in C_c(\partial_L\Omega_T)$, there exists a solution to the Dirichlet problem
\begin{equation*}\label{cylinder}
\left\{\begin{aligned}
\HH u&=0\quad\text{in } \Omega_T,\\
u&=f\quad \text{n.t. on } \partial_L\Omega_T,\\
u&=0\quad\text{on }\partial_P\Omega_T=\overline{\Omega}\times\{t=0\},
\end{aligned}\right.
\end{equation*}
such that 
$$ \|N(u)\|_{L^p(\partial_L\Omega_T)}
	\le C \|f\|_{L^p(\partial_L\Omega_T)}.$$  
The solvability is equivalent to \eqref{eq:revhold} and \eqref{eq:loc}, with $Q_r(x,t)$ replaced by 
\[\Delta_r(X,t) = \widetilde Q_r(X,t)\cap\partial_L\Omega_T,\; (X,t)\in \partial_L\Omega_T\text{ and }(Z,\tau)\in \Omega_T,\] 
and with the measure $d\sigma(X,t)$, see \eqref{measure}, in place of $dxdt$. 

\subsection{Preliminaries}\label{lemmas}

We now recall some well known results that will be needed for the proof of Theorem \ref{per1}. For the Lemmas \ref{Gomega}-\ref{comp} below we refer to \cite{FGS} and the references therein. 
For a time-independent Lipschitz domain $D$ (given either by \eqref{unbounded} or \eqref{cyl}), we denote by $G$ Green's function with respect to $D$, with the convention that $G(X,t;Z,\tau)$ is Green's function with pole at $(Z,\tau)\in D$. Green's function $G=G(\cdot;Z,\tau)$, as a function of $(X,t)$, satisfies 
\begin{align}
&\partial_t G(X,t)+\LL G(X,t) = \delta(X-Z,t-\tau)\quad\text{in }D,\label{Gcal}\\
&G(X,t)=0\quad\text{on }\partial_LD\cup\partial_P D. \label{Gbdry}
\end{align}
Since the operator $\LL$ is symmetric we have $G(X,t;Z,\tau)=G(Z,t;X,\tau)$. Additionally, the time-independence of $A$ implies that $G(X,t;Z,\tau)$ depends only on the time difference $t-\tau$. To see this we note that if the function $v(X,t)$ is $\LL$-caloric, then so is $v(X,t+t_0)$. It follows that $G(X,t+t_0;Z,\tau+t_0)$ satisfies \eqref{Gcal} and \eqref{Gbdry}. Combining the symmetry in space and the time-invariance we obtain 
\begin{equation}
G(X,t;Z,\tau) = G(Z,t+t_0;X,\tau+t_0). \label{rearrange}
\end{equation}
We also recall the estimate 
\begin{equation}\label{greenestimate}
G(X,t;Z,\tau)\le \frac{C}{\|(X-Z,t-\tau)\|^{n+1}}.
\end{equation}
We shall also consider the adjoint Green's function $G^*(X,t)$ with pole at $(Z,\tau)$, given by $$G^*(X,t) = G^*(X,t;Z,\tau)=G(Z,\tau;X,t),$$ which is adjoint $\LL$-caloric as a function of $(X,t)$ for $t<\tau$. 

\begin{Lem}\label{Gomega} 
Let $G$ and $\omega$ be Green's function and the $\LL$-caloric measure of $T_R(x_0,t_0)$ or $\R^{n+2}_+$. Suppose $|(x_0,0)-(x,\lambda)|^2\le A|t-t_0|$, and $(x,t,\lambda)\in T_R(x_0,t_0)$ or $(x,t,\lambda)\in \R^{n+2}_+$. Then there exists a constant $c=c(A)\ge 1$ such that if $t-t_0\ge 4\rho^2$, then 
\begin{equation*}
c^{-1}\rho^{n+1}G(x,t,\lambda;x_0,t_0+\rho^2,\rho)\le \omega(x,t,\lambda,\Delta(x_0,t_0,\rho/2))\le c\rho^{n+1}G(x,t,\lambda;x_0,t_0-\rho^2,\rho),
\end{equation*}
and if $t_0-t\ge 4\rho^2$, 
\begin{equation*}
c^{-1}\rho^{n+1}G(x_0,t_0-\rho^2,\rho;x,t,\lambda)\le \omega^*(x,t,\lambda,\Delta(x_0,t_0,\rho/2))\le c\rho^{n+1}G(x_0,t_0+\rho^2,\rho;x,t,\lambda).
\end{equation*}
\end{Lem}

\begin{Lem}[Harnack's inequality]\label{intharnack}
Let $\Omega$ be a convex domain in $\R^{n+1}. $If $\HH u=0$ in $\Omega\times(t_0,T_0)$ and $u\ge 0$ in $\Omega\times(t_0,T_0)$, then if $(y,\sigma),\; (x,\lambda)\in \Omega$ and $t_0<s<t<T_0$,
\begin{equation*}
u(y,s,\sigma) \le Cu(x,t,\lambda)\exp\left(C\frac{|x-y|^2+|\lambda-\sigma|^2}{t-s}+\frac{t-s}{R}+1\right), 
\end{equation*}
where $R=\min\{\text{dist}(x,\partial \Omega)^2,\text{dist}(y,\partial \Omega)^2,s-t_0,1\}$. 
\end{Lem}

\begin{Lem}\label{harnack}
If $\HH u=0$ in $T_{4r}(x_0,t_0)$ and $u\ge 0$ in $T_{4r}(x_0,t_0)$, then  
\begin{equation*}
u(x,t,\lambda) \le Cu(x_0,t_0+2r^2,r), \quad\text{for all }(x,t,\lambda)\in T_r(x_0,t_0).
\end{equation*}
\end{Lem}

\begin{Lem}\label{comp}
Suppose $u$ and $v$ are non-negative solutions to $\HH u=0$ in $T_{4r}(x_0,t_0)$, continuous in $\overline{T_{4r}(x_0,t_0)}$ and that $u=v=0$ on $Q_{2r}(x_0,t_0)$. 
Then 
\begin{equation*}
\frac{u(x,t,\lambda)}{v(x,t,\lambda)}\le C\frac{u(x_0,t_0+2r^2,r)}{v(x_0,t_0-2r^2,r)}, \quad\text{for all }(x,t,\lambda)\in T_r(x_0,t_0).
\end{equation*}
If $u$ satisfies $\HH^*u=0$ in $T{4r}(x_0,t_0)$, then 
\begin{equation*}
\frac{u(x,t,\lambda)}{v(x,t,\lambda)}\le C\frac{u(x_0,t_0-2r^2,r)}{v(x_0,t_0+2r^2,r)}, \quad\text{for all }(x,t,\lambda)\in T_r(x_0,t_0).
\end{equation*}
\end{Lem}

Let $1 \leq p < \infty$. We say that $u$ is \textit{locally H\"older continuous} in a domain $D$ if there exist constants $C>0$ and $0<\alpha<1$ verifying 
\begin{equation}\label{eq:Holder}
	|u(X,t)-u(Y,s)|
    	\leq C \Big( \frac{\|(X-Y,t-s)\|}{r} \Big)^\alpha \Big( \mean{2\widetilde{Q}} |u|^p \Big)^{1/p},
        \quad (X,t), (Y,s) \in \widetilde{Q},
\end{equation}
for every parabolic cube $\widetilde{Q}:=\widetilde{Q}_r \subset \R^{n+2}$ such that  $2\widetilde{Q} :=\widetilde{Q}_{2r}\subset D$. Moreover, any $u$ satisfying \eqref{eq:Holder} also satisfies \textit{Moser's local estimate} 
\begin{equation}\label{eq:Moser}
	\sup_{\widetilde{Q} }|u|
    	\leq C \Big( \mean{2\widetilde{Q}} |u|^p \Big)^{1/p}.
\end{equation}
By the classical \textit{De Giorgi-Moser-Nash theorem}
(\cite{Nash}) any solution of $\HH u=0$ in 
$2\widetilde{Q}$, verifies both estimates \eqref{eq:Holder} and \eqref{eq:Moser}.
This is true for any real matrix $A$ satisfying \eqref{eq:Aellip}, without extra regularity needed.
Additionally, if $D$ is a time-independent Lipschitz domain and $2\widetilde Q\cap\partial D\neq\emptyset$ and if $u=0$ on $2\widetilde Q\cap\partial D$, then 
\begin{equation}\label{eq:MoserBoundary}
	\sup_{\widetilde{Q}\cap  D }|u|
    	\leq C \Big( \mean{2\widetilde{Q}\cap D} |u|^p \Big)^{1/p},
\end{equation}
and \eqref{eq:Holder} holds for $(X,t), (Y,s)\in \widetilde Q\cap D$. 
It is well known that if \eqref{eq:Holder} or \eqref{eq:Moser} hold for one single value of $p$, then they hold for all $1 \leq p < \infty$. We remark that \eqref{eq:Holder}-\eqref{eq:MoserBoundary} also holds for solutions to $\HH^* u=0$.

Let $\omega$ be the $\LL$-caloric measure of the domain $\R^{n+2}_+$. The nonnegative function 
\[v(X,t) = 1-\omega(X,t,Q_{r}(x_0,t_0))\] is $\LL$-caloric in $\R^{n+2}_+$, vanishes on $Q_{r}\times\{0\}$ and hence is H\"older continuous on $\overline{T_{r/2}(x_0,t_0)}$. It easily follows that there is a constant $0<\gamma<1$ such that $\omega(X,t,Q_{r}(x_0,t_0))\ge \frac12$ if $(X,t)\in T_{\gamma r}(x_0,t_0)$. By Harnack's inequality, there exists $c_0>0$ such that  
\begin{equation}\label{omegapositive}\omega(x,t,\lambda,Q_{r})\ge c_0,\end{equation}
if $(x,t,\lambda)$ satisfies $\lambda>\gamma r$ and $|x-x_0|^2+\lambda^2\le C_1(t-t_0)\le C_2r^2$ for some $C_1$ and $C_2$, with $c_0$ depending on $C_1$ and $C_2$. In view of Lemma \ref{Gomega}, we get that if 
\[\lambda>\gamma r\text{ and }|(x_0,0)-(x,\lambda)|^2\le A(t-t_0)\le 10Ar^2,\] 
then there is a positive constant $c$ such that 
\begin{equation}\label{Gpositive}G(x,t,\lambda;x_0,t_0,r)\ge cr^{-n-1}. \end{equation}


\begin{Lem}\label{RHsolv}
The reverse H\"older inequality holds if and only if the Dirichlet problem is solvable in $L^p$, in the sense of Definition \ref{solvability}. 
\end{Lem}

\begin{proof}
We will prove that if the reverse H\"older inequality holds for $\alpha>2$, then for any $(x,t)\in\R^{n+1}$ and any $r>0$, 
\begin{equation}\label{Nu}
u(x,t,r)\le C(\M(f^\beta)(x,t))^{1/\beta}, 
\end{equation}
uniformly in $(x,t,r)$, where $\frac{1}{\beta}+\frac{1}{\alpha}=1$. We may clearly assume that $f\ge0$ and hence $u\ge 0$. Thus Harnack's inequality implies that $Nu(x,t)\le C(\M(f^\beta)(x,t))^{1/\beta}$. Choose $\delta>0$ so that $2-\delta=\beta$. Then if $p>2-\delta$ and $f\in L^p$, we obtain 
\[
\|Nu\|_{L^p(\R^{n+1})} \le C\|f\|_{L^p(\R^{n+1})},
\]
by the Hardy-Littlewood maximal function estimate. To prove \eqref{Nu}, we write 
\[
u(x,t,r) = \int_{Q_r(x,t)}K(x,t,r;y,s)f(y,s)dyds + \sum_{j=1}^{\infty}\int_{R_j}K(x,t,r;y,s)f(y,s)dyds, 
\]
where $R_j = Q_{2^jr}(x,t)\setminus Q_{2^{j-1}}(x,t)$. By Harnack's inequality, 
\[\int_{Q_r(x,t)}K(x,t,r;y,s)f(y,s)dyds\le \int_{Q_r(x,t)}K(x,t+4r^2,2r;y,s)f(y,s)dyds.\]
From the reverse H\"older inequality we obtain 
\begin{align} \label{eq:h1}
&\int_{Q_r(x,t)}K(x,t+4r^2,2r;y,s)f(y,s)dyds\\
& \qquad \le r^{n+2}C\left(\frac{1}{r^{n+2}}\int_{Q_r(x,t)}K^\alpha(x,t+4r^2,2r;y,s)\right)^{1/\alpha}\left(\frac{1}{r^{n+2}}\int_{Q_r(x,t)}|f|^\beta\right)^{1/\beta}\nonumber \\ 
&\qquad \le C\omega^{(x,t+4r^2,2r)}(Q_r(x,t))(\M(f^\beta)(x,t))^{1/\beta}\le C(\M(f^\beta)(x,t))^{1/\beta}. \nonumber 
\end{align}
Following the proof of \cite[Lemma 2.1]{FS}, it can be shown that 
\[
\int_{R_j}K(x,t,r;y,s)f(y,s)dyds \le c_j\int_{Q_{2^jr}(x,t)}K(x,t+4^{j+1}r,2^{j+1}r;y,s),
\]
for a sequence $c_j$ such that $\sum_{j=1}^\infty c_j<\infty$. Just like in \eqref{eq:h1} 
we can show that 
\[\int_{R_j}K(x,t,r;y,s)f(y,s)dyds\le Cc_j(\M(f^\beta)(x,t))^{1/\beta},\]
which proves \eqref{Nu}. 

To prove the converse, suppose $(x,t,\lambda)$ satisfies
\begin{equation}\label{weakRHdomain} 
|(x_0,\lambda)-(x,0)|\le |t-t_0|^{1/2},\; \lambda\ge 2r.\end{equation}
Let $f\in C_c(\R^{n+1})$ be a function supported in $Q_r(x_0,t_0)$ and $u$ the corresponding solution to the Dirichlet problem with boundary data $f$. Then 
$$u(x,t,\lambda) 
	= \int_{Q_r(x_0,t_0)}K(y,s)f(y,s)dyds,$$ 
where $K=K(x,t,\lambda; \cdot, \cdot)$.
By \eqref{eq:Moser}
\begin{align*}|u(x,t,\lambda)|&\le C\left(\frac{1}{r^{n+3}}\int_{C_r(x,t,\lambda)}|u|^\beta dYdt\right)^{1/\beta}
\le C\left(\frac{1}{r^{n+2}}\int_{Q_r(x,t)}|N(u)|^\beta dydt\right)^{1/\beta}\\
&\le \frac{C}{r^{(n+2)/\beta}}\left(\int_{Q_r(x,t_0)}|f|^\beta dyds\right)^{1/\beta}.
\end{align*}
Taking the supremum over all $f\in C_c(\R^{n+1})$ supported in $Q_r(x,t_0)$ with $L^\beta$-norm equal to $1$, we see that 
\begin{equation}\label{K1}
\left(\int_{Q_r(x_0,t_0)}|K|^\alpha dyds\right)^{1/\alpha}\le \frac{C}{r^{(n+2)/\beta}}.
\end{equation}
If we prove that 
\begin{equation}\label{K2}
\int_{Q_r(x_0,t_0)}K(y,s)dyds = \omega(x,t,\lambda,Q_r(x,t_0))\ge c_0, 
\end{equation}
then \eqref{K1} and \eqref{K2} imply that for all $(x,t,\lambda)$ satisfying \eqref{weakRHdomain},
\begin{equation}\label{weakRH}
\frac{\left(\frac{1}{r^{n+2}}\int_{Q_r(x_0,t_0)}|K(x,t,\lambda;y,s)|^\alpha dyds\right)^{1/\alpha}}{\frac{1}{r^{n+2}}\int_{Q_r(x_0,t_0)}K(x,t,\lambda;y,s)dyds}\le \frac{C}{c_0}.
\end{equation}
Since \eqref{K2} is a consequence of \eqref{omegapositive}, \eqref{weakRH} follows. \\ 

Note that \eqref{weakRH} is an apriori weaker statement than \eqref{eq:revhold} due to the restriction \eqref{weakRHdomain}. However, following the proof of \eqref{Nu}, we see that \eqref{weakRH} is in fact enough to prove \eqref{Nu}. By Lemma 2.1 and Theorem 3.1 in  \cite{FS}, \eqref{Nu} implies the reverse H\"older inequality \eqref{eq:revhold}, so \eqref{weakRH} and \eqref{eq:revhold} are actually equivalent. 
\end{proof}

\begin{Lem}\label{Prop:equiv2.3-2.4}
	Properties \eqref{eq:revhold} and \eqref{eq:loc} are equivalent.
\end{Lem}

\begin{proof}
Assume that the reverse H\"older inequality \eqref{eq:revhold} holds. 
Note that by \eqref{eq:kernelK} and Lemma \ref{Gomega}, 
\[K(Z,\tau;x,t) = \lim_{\lambda\to0}\frac{G(Z,\tau;x,t,\lambda)}{\lambda}.\]
Let $Z=(x_0,5r)$, let $\tau=t_0+20r^2$ and let $G^*(x,t,\lambda) = G(Z,\tau;x,t,\lambda)$. 
We write 
\[\frac{u(x,t,\lambda)}{\lambda} = \frac{u(x,t,\lambda)}{G^*(x,t,\lambda)}\frac{G^*(x,t,\lambda)}{\lambda}.\]
For any $(x,\hat t,\lambda)\in T_r(x_0,t_0)$, we have 
\[
\frac{u(x,\hat t,\lambda)}{G^*(x,\hat t,\lambda)}\le \sup_{|t-t_0|<r^2}
\frac{u(x,t,\lambda)}{G^*(x,2\hat t-t,\lambda)}. 
\]
Since $G^*$ is adjoint caloric, the function $v(x,t,\lambda)=G^*(x,2\hat t-t,\lambda)$ is caloric in $T_{4r}(x_0,t_0)$. Using Lemma \ref{comp}, we see that 
\[\frac{u(x,t,\lambda)}{G^*(x,t,\lambda)}\le C\frac{u(x_0,t_0+2r^2,r)}{G^*(x_0,2t-t_0+2r^2,r)},\quad \text{for all }(x,t,\lambda)\in T_r(x_0,t_0).\]
Additionally, $G^*(x_0,2t-t_0+2r^2,r)\ge cr^{-n-1}$ for all such $t$ by \eqref{Gpositive}. 
Thus, if \eqref{eq:revhold} holds,  
\begin{align*}
&\int_{Q_r(x_0,t_0)} \limsup_{\lambda\to0} \left|\frac{u(x,t,\lambda)}{\lambda}\right|^2dxdt\le 
C \sup_{|t-t_0|<r^2}\left(\frac{u(x_0,t_0+2r^2,r)}{G^*(x_0,2t-t_0+2r^2,r)}\right)^{2}\int_{Q_r(x_0,t_0)}K^2(x,t)dxdt\\
& \qquad \qquad \le Cu^2(x_0,t_0+2r^2,r) r^{2n+2} r^{n+2}\frac{1}{|Q_r|}\int_{Q_r(x_0,t_0)}K^2(x,t)dxdt\\
& \qquad \qquad\le Cu^2(x_0,t_0+2r^2,r)r^{3n+4}\left(\frac{1}{|Q_r|}\int_{Q_r(x_0,t_0)}K(x,t)dxdt\right)^2\\
& \qquad \qquad \le Cu^2(x_0,t_0+2r^2,r)r^{n} \omega^2(Q_r) \le \frac{C}{r^3}\int_{T_{2r}(x_0,t_0)}u^2(x_0,t_0+2r^2,r)dxdtd\lambda\\
&\qquad \qquad \le \frac{C}{r^3}\int_{T_{2r}(x_0,t_0)}u^2(x,t,\lambda)dxdtd\lambda, 
\end{align*}
where the last inequality is a consequence of Harnack's inequality. \\

If \eqref{eq:loc} holds, fix $(x_0,t_0)$ and let $(Z,\tau)$ satisfy $|(x_0,0)-Z|^2\le |t_0-\tau|$, $z_{n+1}\ge 2r$ and $\tau-t_0\ge 16r^2$. Choose $u(x,t,\lambda)=G(Z,\tau;x,t,\lambda)=G^*(x,t,\lambda)$ in \eqref{eq:loc}, then $(-\partial_t+\LL)u=0$ in $T_{4r}(x_0,t_0)$. Thus 
\begin{align}\label{loctorh1}
\int_{Q_r(x_0,t_0)}K^2(x,t)dxdt
&=\int_{Q_r(x,t)} \limsup_{\lambda\to0} \left|\frac{G^*(x,t,\lambda)}{\lambda}\right|^2dxdt\\
&  \le\frac{C}{r^3}\int_{T_{2r}(x_0,t_0)}(G^*)^2(x,t,\lambda)dxdt\lambda, \nonumber
\end{align}
where $K = K^{(Z,\tau)}$. By Lemma \ref{harnack} and Lemma \ref{Gomega}, 
\begin{align*}
&\frac{C}{r^3}\int_{T_{2r}(x_0,t_0)}(G^*)^2(x,t,\lambda)dxdtd\lambda\le \frac{C}{r^3}r^{n+3}(G^*)^2(x_0,t_0-9r^2,r) 
\le Cr^{-n-2}\omega^2(Q_r(x_0,t_0-10r^2))\\ 
&\le Cr^{-n-2}\omega(Q_{11r}(x_0,t_0))\le C\omega(Q_r(x_0,t_0))
\le Cr^{n+2}\left(\frac{1}{|Q_r|}\int_{Q_r(x_0,t_0)}K(x,t) dxdt\right)^{2},
\end{align*}
where we used the doubling property \eqref{doubling}. 
This together with \eqref{loctorh1}
shows that \eqref{eq:loc} implies \eqref{eq:revhold}. \\


\end{proof}

\subsection{Local solvability}\label{subsec:locsolv}

In order to state the next lemma we shall need to introduce some notation. Let $\Omega_T$ be a Lipschitz cylinder as in \eqref{cyl} and let $S=\partial\Omega\times(0,T)$ be its lateral boundary. 
If $(X,t)\in S$ we let $\Gamma^\eta(X,t)$ be a parabolic nontangential cone of opening $\eta$ and vertex $(X,t)$. We choose $\eta$ so that for all $(X,t)\in S$, $\Gamma^\eta(X,t)\cap S=\{(X,t)\}$ in an appropriate system of coordinates. Let 
\[
\Gamma^\eta_r(X,t) = \Gamma^\eta_r(X,t)\cap\{(Y,s)\in \Omega_T:d((Y,s),S)<r\}.
\]
  If $(X,t)\in\Omega_T$, let $d=d((X,t),S)$ be the parabolic distance from $(X,t)$ to $S$ and define 
\[
Q(X,t) = Q(x,t,\lambda) = Q_{d/4}(x,t)\times (\lambda-d/4,\lambda+d/4), 
\] 
where $d = d(x,t,\lambda,S)/4$. If $\HH_1$ and $\HH_2$ are two operators defined by 
\[
\HH_iu := \partial_t u  -\dv(A_i\nabla u), \quad i=1,2,
\]
where $A_i = A_i(x,t,\lambda)$, let 
\[
\eps(x,t,\lambda) := A_1(x,t,\lambda)-A_2(x,t,\lambda), \quad 
\alpha(x,t,\lambda):= \sup_{Q(x,t,\lambda)}|\eps(y,s,\sigma)|. 
\]
\begin{Th}{(\cite[Theorem 6.5]{N1})}\label{transfer}
Suppose that 
\begin{equation*}
\lim_{r\to0^+}\sup_{(X_0,t_0)\in S}\frac{1}{|\Delta_r(X_0,t_0)|}\int_{\Delta_r(X_0,t_0)}\left(\int_{\Gamma^\eta_r(X,t)}\frac{\alpha^2(Y,s)}{d^{n+3}(Y,s,S)}d\sigma(Y,s)\right)d\sigma(X,t)=0.
\end{equation*}
Then, the Dirichlet problem in $\Omega_T$ is solvable for $\HH_1$ if, and only if, it is solvable for $\HH_2$. 
\end{Th}


\begin{Prop}\label{Prop:localsolvab}
Let $A$ be a real and symmetric matrix satisfying 
\eqref{eq:Aellip},
\eqref{eq:Atindep}
and 
\eqref{eq:Dini}.
Then the local solvability condition, \eqref{eq:loc} for $0<r \leq 1$, is satisfied.
\end{Prop}

\begin{proof}
Without loss of generality, it may be assumed that $(x_0,t_0)=(0,0)$. Let $\phi_1(\lambda)$ be a smooth function that satisfies $\phi_1(\lambda)=1$, for $0\le\lambda<4$, and 
$\phi_1(\lambda)=0$, for $\lambda\ge 8$. Take $\phi_2(x)$ another smooth function verifying $\phi_2(x)=1$, for $0\le|x|<4$, and 
$\phi_2(x)=0$, for $|x|\ge 8$. We define the operator 
$$\HH_1 u:= \partial_t u -\dv(A_1(x,\lambda)\nabla u),$$ 
which is given by the matrix
\[
A_1(x,\lambda) := \phi_2(\lambda)[\phi_1(x)A(x,\lambda)+(1-\phi_1(x))I] + (1-\phi_2(\lambda))I,
\]
where $I$ denotes the $(n+1)$--dimensional identity matrix. Observe that $A_1$ is uniformly elliptic. 

Our goal is to prove that the Dirichlet problem for $\HH_1$ is solvable in $T_{12}$ and thus satisfies the local solvability condition \eqref{eq:loc} in that domain (see Proposition \ref{Prop:equiv2.3-2.4}). In particular, this gives us 
\[
\int_{Q_r}\limsup_{\lambda\to0}\left|\frac{u(x,t,\lambda)}{\lambda}\right|^2dxdt\le \frac{C}{r^3}\int_{T_{2r}}|u(x,t,\lambda)|^2dxdtd\lambda,
\]
for all $0<r<1$, whenever $\HH_1 u=0$ in $T_{4r}$ and $u(x,t,0)=0$ on $Q_{4r}$. Notice that, when $\HH u = 0$ in $T_{4r}$, then also $\HH_1 u = 0$ in $T_{4r}$ and thus the local solvability condition for $\HH$ follows from that of $\HH_1$.

We introduce yet another operator $\HH_2$ through the matrix
\[
A_2(x,\lambda) := \phi_2(\lambda)[\phi_1(x)A(x,0) + (1-\phi_1(x))I] + (1-\phi_2(\lambda))I,  
\]
which is easier to handle. 
To prove solvability for $\HH_2$ we are going to show that \eqref{eq:loc} holds. Hence, it is enough to show that in any unit neighborhood $\mathcal{N}:=Q_1(x,t)\times(\lambda-1,\lambda+1)$ of each $(x,t,\lambda)\in S$, the Dirichlet problem for $\HH_2|_{\mathcal N}$ is solvable. 
When $\lambda\le 1$, $A_2(x,\lambda)$ does not depend on $\lambda$, and in this situation the solvability has been established previously in \cite{CNSande} and \cite{N2}.  At unit distance from the remaining part of the boundary, $A_2=I$, for which the solvability is well known.

Next, we make use of Theorem \ref{transfer} to transfer the solvability from $\HH_2$ to $\HH_1$. We have that
\[
\eps(x,t,\lambda):=\eps(x,\lambda):=A_2(x,\lambda)-A_1(x,\lambda) = \phi_2(\lambda)\phi_1(x)(A(x,\lambda)-A(x,0)). 
\]
If $d(x,t,\lambda,S)<1$ and $\lambda>1$, then either $\lambda>10$ or $|x|>10$, which implies $\eps(x,\lambda)=0$. 
If $(Z,\tau)$ belongs to the lateral boundary $S$ of $T_{12}$ and $\rho>0$ is small enough (it suffices to take $\rho<\min(1,1/\eta)$), then if $(x,t,\lambda)\in \Gamma^\eta_\rho(Z,\tau)$, we may have 
$\eps(x,t,\lambda)\neq 0$ only if $(Z,\tau)\in\partial T_{10}\cap\{\lambda=0\}$, in which case $d(x,t,\lambda,S)=\lambda$ and $\lambda<\eta\rho<1$. It follows that 
\begin{equation*}
\alpha(x,t,\lambda):=\alpha(x,\lambda) \le\left\{\begin{aligned}
& |A(x,\lambda)-A(x,0)|,\quad\text{if }(Z,\tau)\in S\cap\{\lambda=0\},\\
&0,\quad\text{otherwise},
\end{aligned}\right. 
\end{equation*}
for all $(x,t,\lambda)\in \Gamma^\eta_\rho(Z,\tau)$ and all $(Z,\tau)\in S$. 
We conclude that if $\rho$ is small enough and $(Z,\tau)\in S$, $(Z,\tau)\not\in \partial T_{10}\cap\{\lambda=0\}$, then 
\[
\int_{\Gamma^\eta_r(Z,\tau)}\frac{\alpha^2(x,t,\lambda)}{d^{n+3}(x,t,\lambda,S)}dxdtd\lambda=0.
\]
If $(Z,\tau)\in \partial T_{10}\cap\{\lambda=0\}$, then 
\begin{align}\label{alambda2}
&\int_{\Gamma^\eta_r(Z,\tau)}\frac{\alpha^2(x,t,\lambda)}{d^{n+3}(x,t,\lambda,S)}dxdtd\lambda \le \int_{\Gamma^\eta_r(Z,\tau)}\frac{|A(x,\lambda)-A(x,0)|^2}{\lambda^{n+3}}dxdtd\lambda \nonumber \\
&\qquad \qquad \le \int_{\Gamma^\eta_r(Z,\tau)}\frac{\theta^2(\lambda)}{\lambda^{n+3}}dxdtd\lambda\le C\int_0^{\eta\rho}\frac{\theta^2(\lambda)}{\lambda}d\lambda\le C\int_0^{1}\frac{\theta^2(\lambda)}{\lambda}d\lambda,
\end{align}
where we used \eqref{eq:Dini} and the fact that the measure of $\Gamma^\eta_r(Q)\cap\{(y,s,\sigma):\sigma=\lambda\}$ is of order $\lambda^{n+2}$. 
As a consequence of \eqref{eq:Dini} we get 
$$\lim_{\rho\to0}\int_0^{\eta\rho} \frac{\theta^2(\lambda)}{\lambda} d\lambda=0.$$
Therefore, since \eqref{alambda2} does not depend on $(Z,\tau)$, the hypothesis of Theorem \ref{transfer} is verified and we conclude that $\HH_1$ is solvable in $T_{12}$, because we already know that $\HH_2$ is solvable in $T_{12}$.  
\end{proof}

\subsection{Local solvability implies \eqref{eq:loc} for all $r>1$}\label{Dahlberg}

%
By localizing the operator $\HH$ we were able to prove local solvability in the previous section. Now, using the periodicity of $A$ we infer \eqref{eq:loc} for all $r>1$. This proof is based on an unpublished work of Dahlberg, which is available in \cite[Appendix]{KS}.
We shall need the following Cacciopolli inequality in the proof. 

\begin{Lem}
\label{Caccioppoli}
Let $R>0$ and for any $\gamma>0$, let $$\Omega_\gamma = \{(x,\lambda): |x_i|<2R\text{ for }i=1\ldots,n,\; 0<\lambda<\gamma R\}.$$ 
Suppose $\HH u = 0$ in $\Omega_4\times(0,16R^2)$ and that $u=0$ on $\partial_L(\Omega_4\times(0,8R^2))\cup \partial_P(\Omega_4\times(0,8R^2))$. 
Then 
\begin{equation}\label{cac}
\int_{\Omega_2\times(0,4R^2)}|\nabla u|^2dxdtd\lambda\le 
\frac{C}{R^2}\int_{\Omega_3\times(0,8R^2)}|u|^2dxdtd\lambda. 
\end{equation}
\end{Lem}
\begin{proof}
Let $\phi(x,t,\lambda) = \phi_1(\lambda)\phi_2(t)$, where $\phi_1$ and $\phi_2$ are smooth cut-off functions such that $\phi_1(\lambda)= 1$ for $|\lambda|\le 2R$, $\phi_1(\lambda)=0$ for $|\lambda|>3R$, $|\phi_1'|\le C/R$ and $\phi_2(t)=1$ for $|t|\le 4R^2$, $\phi_2(t)=0$ for $|t|>8R^2$, $|\phi_2'|\le C/R^2$. The proof then follows by using $u\phi^2$ as a test function in the weak formulation of $\HH u = 0$ in $\Omega_4\times(0,8R^2)$.  
\end{proof}
We remark that \eqref{cac} holds with $\Omega_{\gamma_1}$ and $\Omega_{\gamma_2}$ in place of $\Omega_2$ and $\Omega_3$ for any $0<\gamma_1<\gamma_2<4$, with $C$ depending on $\gamma_1$ and $\gamma_2$. 
The following lemma is a key tool in the proof. 

\begin{Lem}\label{greenest}
Let $R>8$ and let $\Omega_\gamma$ be as in Lemma \ref{Caccioppoli}. Let $A$ be a real and symmetric matrix satisfying 
\eqref{eq:Aellip},
\eqref{eq:Atindep}
and
\eqref{eq:Aperiod}.
Suppose $\HH u=0$ in $\Omega_4\times(0,8R^2)$ and that $u=0$ on $\partial_P(\Omega_4\times(0,8R^2))$ and $u=0$ on $\partial_L(\Omega_4\times(0,8R^2)$. Define 
\[\Q u(x,t,\lambda):=u(x,t,\lambda+1)-u(x,t,\lambda).\]
Then, for $(x,t,\lambda)\in \overline{\Omega_2\times(0,4R^2)}$ such that $\lambda\ge R$, we have 
\begin{equation*}
|\Q u(x,t,\lambda)|\le \frac{C}{R}\left(\frac{1}{R^{n+3}}\int_{\Omega_3\times(0,8R^2)}|u(x,t,\lambda)|^2dxdtd\lambda\right)^{1/2}. 
\end{equation*} 
\end{Lem}


\begin{proof}
By the periodicity of $A$, $\HH \Q u=0$ in $\Omega_3\times(0,8R^2)$. Thus,
for $(x,t,\lambda) \in \Omega_2\times(0,4R^2)$ such that $\lambda\ge R$, \eqref{eq:MoserBoundary} yields
\begin{equation*} 
|\Q u(x,t,\lambda)|\le C\left(\frac{1}{R^{n+3}} \int_{K_R(x,t,\lambda)} |\Q u|^2dydsd\sigma\right)^{1/2}, 
\end{equation*}
where $K_R(x,t,\lambda)=\widetilde Q_{R/4}(x,t,\lambda)\cap (\Omega_4\times(0,8R^2))$. Let 
\[
I=\{x\in\R^n:|x_i|<2R\text{ for }i=1,\ldots,n\}.
\]
An application of the fundamental theorem of calculus, H\"older's inequality and Fubini's theorem leads to
\begin{align*}
& \int_{K_R(x,t,\lambda)}|\Q u|^2dydsd \sigma
     = \int_{K_R(x,t,\lambda)}\Big| \int_{\lambda}^{\lambda +1} \partial_\sigma u(x,t,\sigma) d\sigma \Big|^2dydsd \lambda \\
& \qquad \leq \int_{I\times(0,8R^2)} \int_{3R/4}^{9R/4} \int_{\lambda}^{\lambda +1} |\nabla u(y,s,\sigma) |^2 d\sigma d \lambda dyds 
 \leq \int_{I\times(0,8R^2)} \int_{3R/4}^{1+9R/4}  |\nabla u(y,s,\sigma) |^2 d\sigma dyds \\
& \qquad \leq \int_{I\times(0,8R^2)} \int_{3R/4}^{10R/4}  |\nabla u(y,s,\sigma) |^2 d\sigma dyds 
\leq \frac{C}{R^2}\int_{\Omega_3\times(0,8R^2)} |u(x,t,\sigma) |^2  dyds d\sigma,
\end{align*}
where in the last inequality we also applied Lemma \ref{Caccioppoli}.
\end{proof}

\begin{Th}\label{Prop:2.4forrgeq1}
Let $A$ be a real and symmetric matrix satisfying 
\eqref{eq:Aellip},
\eqref{eq:Atindep}
and
\eqref{eq:Aperiod}.
Assume that \eqref{eq:loc} holds for $0<r\leq 1$. Then \eqref{eq:loc} also holds for all $r>1$. 
\end{Th}

\begin{proof}
For the sake of simplicity we assume that $(x_0,t_0)=(0,0)$ and write $T_r=T_r(0,0)$ and $Q_r = Q_r(0,0)$. 
We need to prove that 
\begin{equation}\label{allr}
\int_{Q_{r}}\limsup_{\lambda\to0}\left(\frac{u(x,t,\lambda)}{\lambda}\right)^2dxdt\le \int_{T_{2r}}|u(x,t,\lambda)|^2dxdtd\lambda,
\end{equation}
for all $u$ such that $\HH u=0$ in $T_{4r}$ and $u=0$ on $Q_{4r}$.
If $r\le 6$ we may cover $Q_r$ by cubes $Q_{1/2}(x_k,t_k)$ and apply the local solvability condition \eqref{eq:loc} for $0<r\le 1$ to each of them to prove \eqref{allr}. Assume $r>6$ and that $\HH u=0$ in $T_{4r}$ and $u=0$ on $Q_{4r}$. We choose a covering $\{Q_{1/2}(x_k,t_k)\}_k$ of $Q_r$ such that $Q_r\subset\bigcup_kQ_{1/2}(x_k,t_k)\subset Q_{r+1}$ and  $\sum_k\chi_{Q_{1/2}(x_k,t_k)}\le C$, where $C$ is independent of $r$.  \\

By hypothesis, we have
\begin{equation*}
I := \int_{Q_{r}}\limsup_{\lambda\to0}\left(\frac{u(x,t,\lambda)}{\lambda}\right)^2dxdt
\leq C
\sum_k\int_{T_1(x_k,t_k)}|u(x,t,\lambda)|^2dxdtd\lambda.
\end{equation*}
Moreover, Lemma \ref{harnack} gives us
$$u(x,t,\lambda)\le Cu(x_k,t_k+2,1), \quad (x,t,\lambda)\in T_1(x_k,t_k).$$ 
Thus,
\[
I\le C\sum_k|u(x_k,t_k+2,1)|^2.
\]
Let $G_1$ be Green's function for $T_{8r}$ with pole at $(0,-10r^2,5r)$ and let $G_2$ be Green's function for 
\[
\{(x,t,\lambda): |x_i|<8r\text{ for }i=1,\ldots,n,\;-64r^2<t<100r^2,\;0<\lambda<20r\},
\]
with pole at $(0,-10r^2,15r)$.
The boundary comparison principle (Lemma \ref{comp}) tells us that 
\[
\frac{u(x_k,t_k+2,1)}{G_i(x_k,t_k+2,1)}\le C\frac{u(0,2r^2,r)}{G_i(0,-2r^2,r)}, \quad i=1,2.
\]
From \eqref{rearrange} and \eqref{Gpositive} we see that
\[
G_i(0,-2r^2,r)\ge cr^{-n-1}. 
\] 
It follows that 
\begin{equation}
I\le Cr^{2n+2}|u(0,2r^2,r)|^2\sum_kG_1(x_k,t_k+2,1)G_2(x_k,t_k+2,1).\label{IleG1G2}
\end{equation}
\quad\\
From \eqref{rearrange}, Harnack's inequality, using the fact that $t_k\le r^2$ and $r>6$ and Lemma \ref{Gomega} we find that 
\begin{align*}
G_1(x_k,t_k+2,1) 
&= G_1(x_k,t_k+2,1;0,-10r^2,5r) 
= G_1(0,10r^2+2t_k+6,5r;x_k,t_k+4,1)\\
&\le CG_1(0,13r^2,5r;x_k,t_k,1)
\le C\omega_1(Q_1(x_k,t_k+4)), 
\end{align*}
where $\omega_1$ is the $\LL$-caloric measure for $T_{8r}$ with respect to $(0,13r^2,5r)$. Applying Harnack's inequality to $G_2$, wee see that $G_2(x_k,t_k+2,1)\le CG_2(x,t,1)$ for all $(x,t)\in Q_1(x_k,t_k+4)$. Going back to \eqref{IleG1G2}, we obtain 
\begin{align*}
I&\le Cr^{2n+2}|u(0,2r^2,r)|^2\sum_k\omega_1(Q_1(x_k,t_k+4))G_2(x_k,t_k+2,1) \\
&\le Cr^{2n+2}|u(0,2r^2,r)|^2\sum_k\int_{Q_1(x_k,t_k+4)}G_2(x,t,1)d\omega_1(x,t)\\
&\le Cr^{2n+2}|u(0,2r^2,r)|^2\int_{Q_{8r}}G_2(x,t,1)d\omega_1(x,t).
\end{align*}
To estimate this last integral we use Lemma \ref{greenest}. Let $\Q G_2(x,t,\lambda) = (x,t,\lambda+1)-(x,t,\lambda)$. Then $\HH \Q G_2=0$ in $T_{8r}$ since the coefficient matrix $A$ is periodic in the $\lambda$ variable. Thus 
\begin{equation}\label{G2omega1}
\Q G_2(0,13r^2,5r) = \int_{\partial_LT_{8r}}\Q G_2 d\omega_1 = \int_{Q_{8r}\times\{8r\}}\Q G_2 d\omega_1 + \int_{Q_{8r}\times\{0\}}G_2(x,t,1)d\omega_1.
\end{equation}
Using Lemma \ref{greenest}, we find that for $(x,t,\lambda)\in Q_r\times[5r8r]$, 
\begin{equation*}
|\Q G_2(x,t,\lambda)|\le \frac{C}{r}\left(\frac{1}{r^{n+3}}\int_{D_r}|G_2|dydsd\sigma \right)^{1/2}. 
\end{equation*}
where 
\[
D_r = \{(x,t,\lambda): |x_i|<8r,\; -64r^2<t<100r^2,\; 0<\lambda<10r\}.
\]

From Lemma \ref{greenest} and the fact that $|G_2|\le Cr^{-n-1}$ in $D_r$, we get that $|\Q G_2(x,t,\lambda)|\le Cr^{-n-2}$ in $Q_{8r}\times[5r,8r]$. Using this in \eqref{G2omega1} yields the estimate 
\begin{equation*}
\int_{Q_{8r}\times\{0\}}G(x,t,1)d\omega_1\le Cr^{-n-2}. 
\end{equation*}
This leads to the estimate 
\begin{equation*}
I\le Cr^n|u(0,2r^2,r)|^2. 
\end{equation*}
An application of \eqref{eq:Moser} finishes the proof: 
\[
I\le Cr^n|u(0,2r^2,r)|^2\le Cr^n\frac{1}{r^{n+3}}\int_{\widetilde Q_{r/2}(0,2r^2,r)}|u(x,t,\lambda)|^2dxdtd\lambda \le \frac{C}{r^3}\int_{T_{2r}}|u(x,t,\lambda)|^2dxdtd\lambda.
\]

\end{proof}

\subsection{Solvability}

As a consequence of Proposition \ref{Prop:localsolvab}, Theorem \ref{Prop:2.4forrgeq1} and Proposition \ref{Prop:equiv2.3-2.4}, we know that the reverse H\"older inequality \eqref{eq:revhold} holds. Thus the following proposition follows now directly from Lemma \ref{RHsolv}.

\begin{Prop}\label{Prop:solvability}
Suppose that $A$ is a real and symmetric matrix satisfying \eqref{eq:Aellip} -- \eqref{eq:Dini}. Let $f \in C_c(\R^{n+1})$.
Then, there exists $0<\delta<1$ (which depends only in the dimension $n$ and the constants appearing in \eqref{eq:Aellip} and \eqref{eq:loc}) such that the solution to the classical Dirichlet problem
\begin{equation*}
\left\{\begin{aligned}
\HH u&=0\quad\text{in } \R^{n+2}_+,\\
u&=f\quad \text{n.t. on } \R^{n+1},\\
\end{aligned}\right.
\end{equation*}
verifies, for any  $2-\delta<p<\infty$,
$$\|N(u)\|_{L^p(\R^{n+1})}
	\leq C \|f\|_{L^p(\R^{n+1})}.$$
\end{Prop}

\subsection{Uniqueness}
Moving forward to the proof of Theorem \ref{per1}, we start by showing that a solution to 
\begin{equation}\label{Lpnt}
\left\{\begin{array}{l}
\HH u=0\text{ in }\R^{n+2}_+,\\
u=f\text{ n.t on }\partial\R^{n+2}_+=\R^{n+1},\\
\|N(u)\|_{L^p(\R^{n+1})}\le C\|f\|_{L^p(\R^{n+1})},
\end{array}\right.
\end{equation}
where $f\in L^p(\R^{n+1})$ and $p>1$, is unique. The proof relies on the following lemma. 
\begin{Lem}\label{Lem:uvCacciop}
Let $u$, $v$ be weak solutions to $\HH(u)=0$ and $\HH^*(v)=0$ in 
$Q_{2R}(0,0) \times (r/8,4r)$, 
for certain $R \geq r>0$, such that at least one of the solutions is nonnegative. Then,
\begin{align*}
& 
\int_r^{2r} \int_{Q_{R}(0,0)}  \Big(|\nabla u (y,s,\sigma)| \, |v (y,s,\sigma)| + |u (y,s,\sigma)| \, |\nabla v (y,s,\sigma)|\Big) \, dy ds d\sigma \\
& \qquad  \qquad \leq  \frac{C}{r}
\int_{\frac{r}{8}}^{4r} \int_{Q_{2R}(0,0)}  
|u (y,s,\sigma)| \, |v (y,s,\sigma)| \, dy ds d\sigma.
\end{align*}
\end{Lem}

\begin{proof}
	Suppose that $u \geq 0$, the case of $v \geq 0$ follows analogously. It is possible to take points 
    $(x_j, t_j,\lambda_j) \in Q_{R}(0,0) \times (r,2r)$, $j=1, \dots, N$,
    such that
    $$Q_{R}(0,0) \times (r,2r) 
    \subset \bigcup_{j=1}^N \widetilde{Q}_{r/4}(x_j, t_j,\lambda_j)
    \quad \text{and} \quad
    \bigcup_{j=1}^N \widetilde{Q}_{r}(x_j, t_j,\lambda_j)
    \subset Q_{2R}(0,0) \times (r/8,4r).$$
    Then, an application of H\"older's inequality, Cacciopoli's inequality and \eqref{eq:Moser} yields 
    \begin{align*}
    	&\int_r^{2r} \int_{Q_{R}(0,0)} |\nabla u| \, |v|  \, dy ds d\sigma
        	 \leq \sum_{j=1}^N \int_{\widetilde{Q}_{r/4}(x_j, t_j,\lambda_j)} |\nabla u| \, |v|  \, dy ds d\sigma \\
            & \qquad \qquad  \leq \sum_{j=1}^N \Big(\int_{\widetilde{Q}_{r/4}(x_j, t_j,\lambda_j)} |\nabla u|^2 \, dy ds d\sigma \Big)^{1/2} \Big(\int_{\widetilde{Q}_{r/4}(x_j, t_j,\lambda_j)} |v|^2 \, dy ds d\sigma \Big)^{1/2} \\
            & \qquad \qquad  \leq C \sum_{j=1}^N \frac{1}{r} \Big(\int_{\widetilde{Q}_{r/2}(x_j, t_j,\lambda_j)} |u|^2 \, dy ds d\sigma \Big)^{1/2} \Big(\int_{\widetilde{Q}_{r/4}(x_j, t_j,\lambda_j)} |v|^2 \, dy ds d\sigma \Big)^{1/2} \\
            & \qquad \qquad  \leq C \frac{r^{n+3}}{r} \sum_{j=1}^N  \Big(\sup_{\widetilde{Q}_{r/2}(x_j, t_j,\lambda_j)} u \Big)\Big(\sup_{\widetilde{Q}_{r/4}(x_j, t_j,\lambda_j)} |v|\Big) \\
            & \qquad \qquad  \leq  \frac{C}{r} \sum_{j=1}^N  \Big(\sup_{\widetilde{Q}_{r/2}(x_j, t_j,\lambda_j)} u  \Big)\int_{\widetilde{Q}_{r/2}(x_j, t_j,\lambda_j)} |v| \, dy ds d\sigma \\
            & \qquad \qquad  \leq  \frac{C}{r} \sum_{j=1}^N \int_{\widetilde{Q}_{r}(x_j, t_j,\lambda_j)} u |v| \, dy ds d\sigma 
            \leq \frac{C}{r} \int_{\frac{r}{8}}^{4r} \int_{Q_{2R}(0,0)}    u |v| \, dy ds d\sigma,
    \end{align*}
    where in the penultimate step we also used Harnack's inequality (Lemma \ref{intharnack}).
\end{proof}

The following proposition implies uniqueness since the difference of two solutions to \eqref{Lpnt} satisfies its hypothesis.

\begin{Prop}\label{Th:Uniq}
Let $u$ be a weak solution of $\HH u=0$ in $\R^{n+2}_+$ such that $N(u) \in L^p(\R^{n+1})$, for certain $1<p<\infty$, and 
\begin{equation}\label{eq:conv0}
	u(x,t,\lambda) \longrightarrow 0, \quad \text{as } \lambda \to 0^+, \ \text{for a.e. } x\in \R^n, \ t \in \R.
\end{equation}
Assume also that $K(Z,\tau;\cdot) \in L^{p'}(\R^{n+1})$ for all $(Z,\tau) \in \R^{n+2}_+$, where $p'$ is conjugate to $p$.
Then,
$u \equiv 0$ in $\R^{n+2}_+$.
\end{Prop}

\begin{proof}
Fix $(Z,\tau) \in \R^{n+2}_+$ and let $G^*(X,t;Z,\tau)$ be Green's function related to the adjoint operator $\HH^*=-\partial_t+\LL$  on $\R^{n+2}_+$ with pole at $(Z,\tau)$. For each $\ell \in \N$, we take the following auxiliary functions:
\begin{itemize}
	\item $\varphi \in C_c^\infty(\R^n)$, s.t. $\supp(\varphi) \subset B(0,\ell/2)$, $\varphi \equiv 1$ in $ B(0,\ell/4)$ and $|\nabla \varphi| \leq C 1/\ell$;
    \item $\phi \in C_c^\infty(\R)$, s.t. $\supp(\phi) \subset (-\ell^2/2,\ell^2/2)$, $\phi \equiv 1$ in $ (-\ell^2/4,\ell^2/4)$ and $|\phi'| \leq C 1/\ell^2$;
    \item $\psi \in C_c^\infty(\R)$, s.t. $\supp(\psi) \subset (1/(2\ell),2\ell)$, $\psi \equiv 1$ in $(1/\ell,\ell)$, $|\psi'| \leq C \ell$ in $(1/(2\ell),1/\ell)$ and $|\nabla \psi| \leq C 1/\ell$ in $(\ell,2\ell)$.
\end{itemize}
Then, for $\ell \in \N$ big enough, we can write
\begin{align*}
	u(Z,\tau)
    	& = - \int_{\R^{n+2}_+} \Big[\partial_s G^*(Y,s;Z,\tau) + \dv_Y \Big( A(Y) \cdot \nabla_Y G^*(Y,s;Z,\tau) \Big) \Big] u(Y,s) \varphi(y) \psi(\sigma) \phi(s) \,  dY ds \\
        & = \int_{\R^{n+2}_+} G^* \, \varphi \,  \psi \, \partial_s(u  \, \phi) \, dY ds 
        + \sum_{i,j=1}^{n+1} \int_{\R^{n+2}_+} \phi  \, a_{i,j} \,  \partial_{Y_j}G^* \,  \partial_{Y_i} ( u \, \varphi \, \psi ) dY ds  \\
        & = \int_{\R^{n+2}_+} G^* \, \varphi \,  \psi \, \partial_s u \, \phi \, dY ds 
        + \int_{\R^{n+2}_+} G^* \, \varphi \,  \psi \, u \,  \phi' \, dY ds \\
        & \qquad + \sum_{i,j=1}^{n+1} \int_{\R^{n+2}_+} \phi  \, a_{i,j} \,  \partial_{Y_j}G^* \,  u \, \partial_{Y_i} (\varphi \, \psi ) dY ds
        - \sum_{i,j=1}^{n+1} \int_{\R^{n+2}_+} \phi \, G^* \, \partial_{Y_j}\Big( a_{i,j} \,    \partial_{Y_i} u  \, \varphi \, \psi \Big) \, dY ds \\
        & = \int_{\R^{n+2}_+} G^* \, \varphi \,  \psi \, u \, \phi' \, dY ds 
        + \sum_{i,j=1}^{n+1} \int_{\R^{n+2}_+} \phi  \, a_{i,j} \,  \partial_{Y_j}G^* \,  u \, \partial_{Y_i} (\varphi \, \psi ) dY ds \\
       & \qquad  - \sum_{i,j=1}^{n+1} \int_{\R^{n+2}_+} \phi \, G^* \,  a_{i,j} \,    \partial_{Y_i} u  \, \partial_{Y_j}(\varphi \, \psi) \, dY ds,
\end{align*}
where $Y=(y,\sigma)$ with $y \in \R^n$ and $\sigma>0$. Hence,
\begin{align*}
	|u(Z,\tau)|
    	& \leq C \Big(\frac{1}{\ell^2} \int_{\frac{1}{2\ell}}^{2\ell} \int_{\frac{\ell^2}{4}<|s|<\frac{\ell^2}{2}} \int_{|y|<\frac{\ell}{2}} |G^*|  \, |u| \,   dy ds d\sigma \\
        & \qquad + \ell \int_{\frac{1}{2\ell}}^{\frac{1}{\ell}} \int_{|s|<\frac{\ell^2}{2}} \int_{|y|<\frac{\ell}{2}} (|\nabla G^*|  \, |u| + |G^*|  \, |\nabla u| ) \,  dy ds d\sigma \\
        & \qquad + \frac{1}{\ell} \int_{\ell}^{2\ell} \int_{|s|<\frac{\ell^2}{2}}  \int_{|y|<\frac{\ell}{2}} (|\nabla G^*|  \, |u| + |G^*|  \, |\nabla u| ) \,  dy ds d\sigma \\
        & \qquad + \frac{1}{\ell} \int_{\frac{1}{\ell}}^{\ell} \int_{|s|<\frac{\ell^2}{2}} \int_{\frac{\ell}{4}<|y|<\frac{\ell}{2}} (|\nabla G^*|  \, |u| + |G^*|  \, |\nabla u| ) \,  dy ds d\sigma \Big)\\
        &\qquad =: I_1+I_2+I_3+I_4.
\end{align*}
Next, an application of Lemma \ref{Lem:uvCacciop} gives us
\begin{equation*}
I_3\le \frac{C}{\ell^2}\int_{\ell/8}^{4\ell}\int_{|s|<\frac{\ell^2}{2}}  \int_{|y|<\frac{\ell}{2}}|u||G^*| \,  dy ds d\sigma.
\end{equation*}
By \eqref{greenestimate},
\begin{equation}\label{eq:estG*}
	G^*(Y,s;Z,\tau)
    	\leq  \frac{C}{(|Y-Z| + |s-\tau|^{1/2})^{n+1}}
        \leq  \frac{C}{\ell^{n+1}}, 
\end{equation}
when 
$\ell^2/4 <|s|<\ell^2$, 
$\ell/2 < \sigma < 4 \ell$ or
$\ell/8 <|y|<\ell$,
provided that  $\ell$ is sufficiently large.
Hence,
\begin{align}\label{eq:comb2}
I_1(Z,\tau) + I_3(Z,\tau)
	& \leq  \frac{C}{\ell^{n+2}}  \int_{|s|<\ell^2}  \int_{|y|<\ell}  |N(u)(y,s)| \,  dy ds \nonumber \\
    & \leq \frac{C }{\ell^{(n+2)/p}}  \|N(u)\|_{L^p(\R^{n+1})}
    \longrightarrow 0, \quad \text{as } \ell \to \infty.
\end{align}
On the other hand,
\begin{align*}
I_2(Z,\tau) 
    & \leq C \int_{|s|<\ell^2} \int_{|y|<\ell} \M_{2/\ell}(u)(y,s) \Big(\frac{1}{1/\ell}\int_{\frac{1}{4\ell}}^{\frac{2}{\ell}} \frac{G^*(y,s,\sigma;Z,\tau)}{\sigma} \, d\sigma \Big)  dy ds \\
    & \leq C \|\M_{2/\ell}(u)\|_{L^p(\R^{n+1})} \Big\| \frac{1}{1/\ell}\int_{\frac{1}{4\ell}}^{\frac{2}{\ell}} \frac{G^*(y,s,\sigma;Z,\tau)}{\sigma} \, d\sigma \Big\|_{L^{p'}(\R^{n+1})},
\end{align*}
where $\M_{r}(u)$ denotes the truncated vertical maximal function given by
$$\M_{r}(u)(x,t)
	:= \sup_{0<\lambda<r} |u(x,t,\lambda)|.$$
    Since 
    \[
    \frac{G^*(y,s,\sigma;Z,\tau)}{\sigma} = \frac{G(Z,\tau;y,s,\sigma)}{\sigma}, 
    \]
    we have 
    \[
    \lim_{\sigma\to 0}\frac{G^*(Z,\tau;y,s,\sigma)}{\sigma} = K(Z,\tau;y,s),
    \]
    by  Lemma \ref{Gomega} and the definition of $K$. 
By the Lebesgue differentiation theorem, we deduce that
\begin{equation*}
	\lim_{\ell \to \infty} \Big\| \frac{1}{1/\ell}\int_{\frac{1}{4\ell}}^{\frac{2}{\ell}} \frac{G^*(y,s,\sigma;Z,\tau)}{\sigma} \, d\sigma \Big\|_{L^{p'}(\R^{n+1})}
    \leq C \| K(Z,\tau;\cdot)\|_{L^{p'}(\R^{n+1})}<\infty.
\end{equation*}
Moreover, since $\M_{2/\ell}(u)\leq N(u) \in L^p(\R^{n+1})$, the assumption \eqref{eq:conv0} implies
\begin{equation}\label{eq:comb3}
I_2(Z,\tau) \longrightarrow 0, \quad \text{as } \ell \to \infty. 
\end{equation}
To estimate $I_4$, we write 
\begin{equation}
I_4 = \frac{C}{\ell} \sum_{j=0}^N\int_{\frac{2^j}{\ell}}^{\frac{2^{j+1}}{\ell}} \int_{|s|<\frac{\ell^2}{2}} \int_{\frac{\ell}{4}<|y|<\frac{\ell}{2}} (|\nabla G^*|  \, |u| + |G^*|  \, |\nabla u| ) \,  dy ds d\sigma,
\end{equation}
where $N=\log_2 l^2$. By Lemma \ref{Lem:uvCacciop} and \eqref{eq:estG*}, 
\begin{align}\label{I41}
I_4
&=\frac{C}{\ell} \sum_{j=0}^N\int_{\frac{2^j}{\ell}}^{\frac{2^{j+1}}{\ell}} \int_{|s|<\frac{\ell^2}{2}} \int_{\frac{\ell}{4}<|y|<\frac{\ell}{2}} (|\nabla G^*|  \, |u| + |G^*|  \, |\nabla u| ) \,  dy ds d\sigma\\
&\le C\sum_{j=0}^N2^{-j}\int_{\frac{2^j}{8\ell}}^{\frac{2^{j+2}}{\ell}} \int_{|s|<\frac{\ell^2}{2}} \int_{\frac{\ell}{4}<|y|<\frac{\ell}{2}}|G^*||u|dydsd\sigma \nonumber\\
&\le C\sum_{j=0}^N\frac{1}{\ell^{n+2}}\int_{|s|<\frac{\ell^2}{2}} \int_{\frac{\ell}{4}<|y|<\frac{\ell}{2}}N(u)dyds \nonumber \\
&\le CN\frac{1}{\ell^{n+2}}\int_{Q_{\ell}}N(u)dyds
\le \frac{C\log_2\ell^2}{\ell^{\frac{n+2}{p}}}\|N(u)\|_{L^p(\R^{n+1})}\to0,
\text{ as }\ell\to\infty.\nonumber 
\end{align}
Therefore, since $(Z,\tau)$ was taken arbitrary in $\R^{n+2}_+$, putting together \eqref{eq:comb2},  \eqref{eq:comb3} and \eqref{I41}
, we conclude $u \equiv 0$ in $\R^{n+2}_+$.
\end{proof}

\subsection{Proof of Theorem~\ref{per1}}

\begin{proof}[Proof of Theorem~\ref{per1}]
	Let $f \in L^p(\R^{n+1})$, with $2-\delta<p<\infty$; where $0<\delta<1$ was determined in Proposition \ref{Prop:solvability}. We can take functions $\{f_k\}_{k \in \N} \subset C_c(\R^{n+1})$ such that $f_k \longrightarrow f$, $k \to \infty$, in $L^p(\R^{n+1})$. Then, for each $k \in \N$, call $u_k$ the solution provided in Proposition \ref{Prop:solvability} with boundary data $f_k$, which satisfies the estimate
$$\|N(u_k)\|_{L^p(\R^{n+1})}
	\leq C \|f_k\|_{L^p(\R^{n+1})}.$$   
We also have that
$$\|N(u_j-u_k)\|_{L^p(\R^{n+1})}
	\leq C \|f_j - f_k\|_{L^p(\R^{n+1})}, \quad j,k \in \N,$$ 
and from here we infer that there exists a function $u$ such that $u_k \longrightarrow u$, $k \to \infty$, uniformly on compact sets of $\R_+^{n+2}$. Moreover, standard arguments
guarantee that $u$ is a weak solution of the
Dirichlet problem 
\begin{equation*}
\left\{\begin{aligned}
\HH u&=0\quad\text{in } \R_+^{n+2},\\
u&=f\quad \text{n.t. on } \R^{n+1},\\
\end{aligned}\right.
\end{equation*}
verifying     
$$\|N(u)\|_{L^p(\R^{n+1})}
	\leq C \|f\|_{L^p(\R^{n+1})}.$$ 
For the fact that $u=f$ n.t on $\R^{n+1}$ we refer to \cite{FGS}. 
On the other hand, the uniqueness is a consequence of Proposition \ref{Th:Uniq}, since the kernel $K(Z,\tau;y,s) \in L^{p'}(\R^{n+1})$, for all $(Z,\tau)=(z,\sigma,\tau) \in \R^{n+2}_+$.
Indeed, by duality, 
\begin{align*}	
\|K(Z,\tau;\cdot) \|_{L^{p'}(\R^{n+1})}
& =  \sup_{g} \Big| \int_{\R^{n+1}} K(Z,\tau;x,t)g(x,t) dx dt \Big| 
=  \sup_{g} |v_g(Z,\tau)| \\
& \leq C  \sup_{g} \Big( \mean{\widetilde{Q}_{\sigma/2}(Z,\tau)}  |v_g|^p \Big)^{1/p} 
\leq C \sigma^{-(n+3)/p} \sup_{g} \|N(v_g)\|_{L^p(\R^{n+1})} \\
& \leq C  \sigma^{-(n+3)/p} \sup_{g} \|g\|_{L^p(\R^{n+1})}
\leq C \sigma^{-(n+3)/p} 
< \infty.
\end{align*}
Here the supremum was taken over all $g \in C_c(\R^{n+1})$ such that $\|g\|_{L^p(\R^{n+1})} \leq 1$; $v_g$ is the solution to the Dirichlet problem with boundary data $g$ 
and in the third inequality we used \eqref{eq:Moser}.
\end{proof}

\section{Homogenization}

We divide the proof of Theorem \ref{th:homogenization} in three steps.
\subsection{Proof of \eqref{eqeps} and \eqref{nteps} for $D$}
 By making the change of variables 
$(x,t,\lambda)\mapsto(y,s,\sigma)$ given by $(x,t,\lambda)=(\eps y,\eps^2s,\eps\sigma)$, the boundary 
$$\partial D=\{(x,t,\lambda)=(x,t,\phi(x))\}$$
is transformed into 
$$\partial D_\eps := \{(y,s,\sigma)=(y,s,\phi_\eps(y))\},$$ where 
$\phi_\eps(y):=\eps^{-1}\phi(\eps y)$. 
Note that $\phi$ and $\phi_\eps$ have the same Lipschitz constant. \\

Let 
$$v_\eps(y,t,\sigma) := u_\eps(\eps y,\eps^2s,\eps\lambda) \quad 
and 
\quad 
f_\eps(y,s,\phi_\eps(y)) := f(\eps y,\eps^2s,\phi(\eps y)).$$
Then,
\begin{equation}\label{epsscale}
\left\{\begin{aligned}
 \partial_t u_\eps  +\LL_\eps u_\eps & = 0 \quad \text{in }D,\\
u_\eps &= f \quad \text{n.t. on }\partial D,
\end{aligned}\right.
\end{equation}
holds if, and only if, 
\begin{equation}\label{rescale}
\left\{\begin{aligned}
\partial_s v_\eps + \LL v_\eps & = 0 \quad \text{in }D_\eps,\\
v_\eps &=f_\eps \ \  \text{n.t. on }\partial D_\eps.
\end{aligned}\right.
\end{equation}
 By Theorem \ref{per1}, \eqref{rescale} has a unique solution that satisfies 
$$\|N(v_\eps)\|_{L^2(\partial D_\eps)}
\le C\|f_\eps\|_{L^2(\partial D_\eps)}.$$ 
Changing back to the $(x,t,\lambda)$ coordinates, we get that \eqref{epsscale} has a unique solution verifying the estimate 
$$\|N(u_\eps)\|_{L^2(\partial D)}\le C\|f\|_{L^2(\partial D)}.$$

\subsection{Proof of \eqref{eqeps} and \eqref{nteps} for $\Omega_T$}
We are going to prove that the kernel $K_\eps$ associated to the caloric measure $\omega_\eps$ for $\partial_t + \LL_\eps$ on $\partial_L\Omega_T$ satisfies the reverse H\"older inequality. \\

Let $(x_0,t_0,\lambda_0)\in \partial_L \Omega_T$. Then, after rotating the coordinates if necessary, one has by \eqref{localbdry} 
\begin{equation}\label{locallip}
\Omega_T\cap U(x_0,t_0,\lambda_0) = \{(\tilde x,\tilde t,\tilde \lambda):\tilde \lambda>\phi(\tilde x)\}\cap U(x_0,t_0,\lambda_0). 
\end{equation} 
In the new (rotated) coordinates $(\tilde x, t, \tilde\lambda)$, $\tilde u(\tilde x,t,\tilde\lambda)=u(x,t,\lambda)$ solves a parabolic equation of the same type, 
$$\partial_t \tilde u -\text{div}(\tilde A\nabla \tilde u)=0,$$
but in general $\tilde A$ will not be periodic in $\tilde \lambda$.\\ 

Suppose that the representation of $(\tilde x,\tilde \lambda)$ in the original coordinates is given by $\tilde\lambda = l\nu$, $\nu\in\R^{n+1},\;|\nu|=1$, and $\tilde x=\hat x$ in the $(x,\lambda)$ coordinates. Then, $\tilde A$ has period $\tilde\lambda_0=l_0\nu$ in $\tilde\lambda$ if and only if 
\begin{equation}\label{periodicity}
A(\hat x+(l+l_0)\nu) = \tilde A(\tilde x,\tilde\lambda+\tilde\lambda_0) = \tilde A(\tilde x,\tilde\lambda) = A(\hat x+l\nu). 
\end{equation}
From the periodicity of $A$ we see that \eqref{periodicity} holds if and only if $l_0\nu\in \Z^{n+1}$. Since $|\nu|=1$ this is equivalent to 
\begin{equation}\label{normal}
\nu=\frac{\nu_0}{|\nu_0|}, \quad \nu_0\in\Z^{n+1}\setminus\{0\}.
\end{equation}
However, since $\phi$ is Lipschitz, there is room to rotate the coordinates further to obtain \eqref{normal}, while maintaining the representation \eqref{locallip}. Thus we may assume that $\tilde A$ is periodic in $\tilde\lambda$. \\

We extend $\phi$ to $\R^n$, preserving its Lipschitz norm, and let $D=\{(x,t,\lambda):\lambda >\phi(x)\}$. Denote by $K^D_\eps$ the kernel associated to $D$, with respect to $(z,\tau,l)\in\Omega_T\cap U_{(x_0,t_0,\lambda_0)}$ such that $\tau-t_0\ge 4r^2$ and $|(z,l)-(x_0,\lambda_0)|^2\le \tau-t_0$. 
From the first part of the proof we know that the Dirichlet problem for $\partial_t+\LL_\eps$ in $D$ is solvable in $L^2$. Thus $K^D_\eps$ satisfies the reverse H\"older inequality, by Lemma \ref{RHsolv}.\\

Let $K^{\Omega_T}_\eps$ be the kernel associated to $\Omega_T$, with respect to $(z,\tau,l)$. We need to show that 
\[
\left(\frac{1}{r^{n+2}}\int_{\Delta_r}|K_\eps^{\Omega_T}|^2d\sigma\right)^{1/2}\le\frac{C}{r^{n+2}}\int_{\Delta_r}|K_\eps^{\Omega_T}|d\sigma,
\]
for any 
\begin{equation*}
\Delta_r := \Delta_r(x_0,t_0,\lambda_0), \quad r<r_0. 
\end{equation*}
We recall that the measure $\sigma$ was defined in \eqref{measure}. 
Let $G^{\Omega_T}_\eps$ be Green's function for $\Omega_T$ and let $G^D_\eps$ be Green's function for $D$. We denote by $G^{*\Omega_T}_\eps$ and $G^{*D}_\eps$ the corresponding adjoint Green's functions with pole at $(z,\tau,l)$. 

If $\Delta_\lambda(\hat x,\hat t,\hat \lambda)\subset\Delta_r$ and $\lambda>0$ is small enough, 
\begin{align*}
\frac{\omega_\eps^{\Omega_T}(\Delta_\lambda)}{\lambda^{n+2}}&\le C\frac{G^{*\Omega_T}_\eps(\hat x,\hat t- 4\lambda^2,\hat \lambda+2\lambda)}{\lambda}\\
&=C\frac{G_\eps^{*D}(\hat x,\hat t- 4\lambda^2,\hat \lambda+2\lambda)}{\lambda}\frac{G^{*\Omega_T}_\eps(\hat x,\hat t-4\lambda^2,\hat \lambda+2\lambda)}{G^{*D}_\eps(\hat x,\hat t- 4\lambda^2,\hat \lambda+2\lambda)}\\
&\le C\frac{G_\eps^{*D}(\hat x,\hat t- 4\lambda^2,\hat \lambda+2\lambda)}{\lambda}\frac{G^{*\Omega_T}_\eps(x_0, t_0-2r^2, \lambda_0+r)}{G^{*D}_\eps(x_0,t_0+2r^2,\lambda_0+r)}\\
&\le  C\frac{\omega^{D}_\eps(\Delta_\lambda(\hat x,\hat t-8\lambda^2,\phi(\hat x,\hat t-8\lambda^2)))}{\lambda^{n+2}}
\frac{G^{*\Omega_T}_\eps(x_0,t_0-2r^2,\lambda_0+r)}{G^{*D}_\eps(x_0,t_0+2r^2,\lambda_0+r)},
\end{align*}
where we used \eqref{Gomega} and \eqref{comp}. 
Taking $\lambda\to0$, it follows that 
\[
K_\eps^{\Omega_T}\le CK_\eps^D\frac{G^{*\Omega_T}_\eps(x_0,t_0-2r^2,\lambda_0+r)}{G^{*D}_\eps(x_0,t_0+2r^2,\lambda_0+r)}. 
\]
Since $K_\eps^D$ satisfies the reverse H\"older inequality,
\begin{align*}
\left(\frac{1}{r^{n+2}}\int_{\Delta_r}|K_\eps^{\Omega_T}|^2d\sigma\right)^{1/2}&\le C\frac{G^{*\Omega_T}_\eps(x_0,t_0-2r^2,\lambda_0+r)}{G^{*D}_\eps(x_0,t_0+2r^2,\lambda_0+r)}
\left(\frac{1}{r^{n+2}}\int_{\Delta_r}|K_\eps^D|^2d\sigma\right)^{1/2}\\
&  \le C\frac{\omega_\eps^{*D}(\Delta_r)}{r^{n+2}}\frac{G^{*\Omega_T}_\eps(x_0,t_0-2r^2,\lambda_0+r)}{G^{*D}_\eps(x_0,t_0+2r^2,\lambda_0+r)}\\
&\le C\frac{G_\eps^{*D}(x_0,t_0-2r^2,\lambda_0+r)}{r}\frac{G^{*\Omega_T}_\eps(x_0,t_0+2r^2,\lambda_0+r)}{G^{*D}_\eps(x_0,t_0+2r^2,\lambda_0+r)}. 
\end{align*}
Using Corollary 2.3. in \cite{FGS}, we see that 
\[
\frac{G_\eps^{*D}(x_0,t_0-2r^2,\lambda_0+r)}{G^{*D}_\eps(x_0,t_0+2r^2,\lambda_0+r)}\le C.
\]
Whence, using Lemma \ref{Gomega} and the doubling property  \eqref{doubling}, we obtain
\begin{align*}
\left(\frac{1}{r^{n+2}}\int_{\Delta_r}|K_\eps^{\Omega_T}|^2d\sigma\right)^{1/2}& \le C\frac{\omega^{\Omega_T}(\Delta_r(x_0,t_0-4r^2,\phi(x_0,t_0-4r^2)))}{r^{n+2}} \\
&\le C\frac{\omega^{\Omega_T}(\Delta_{6r}(x_0,t_0,\lambda_0))}{r^{n+2}}\le C\frac{\omega^{\Omega_T}(\Delta_{6r}(x_0,t_0,\lambda_0))}{r^{n+2}}\\
&\le \frac{1}{r^{n+2}}\int_{\Delta_r}K_\eps^{\Omega_T}d\sigma.
\end{align*}
 Thus $K_\eps^{\Omega_T}$ satisfies the reverse H\"older inequality, which proves \eqref{eqeps} and \eqref{nteps} for $\Omega_T$, using once again Lemma \ref{RHsolv}.

\quad
\subsection{Proof of \eqref{homogeneous}}
We now turn to the homogenization result.
Since the domain $\Omega_T$ is bounded, the $L^p$ norm of $u_\eps$ in $\Omega_T$ can be estimated by the $L^p$ norm of its non tangential maximal function: 
\[
\|u_\eps\|_{L^p(\Omega_T)}\le C(\text{diam}(\Omega_T))\|N(u_\eps)\|_{L^p(\partial_L\Omega_T)}\le C(\text{diam}(\Omega_T))\|f\|_{L^p(\partial_L\Omega_T)}.
\]
Let $\widetilde Q_r$ be a parabolic cube in $\R^{n+2}$ of size $r$ such that $\text{dist}(\widetilde Q_r,\partial D)\ge r$. From the De-Giorgi-Moser-Nash estimate \eqref{eq:Moser}, it follows that 
\begin{equation}\label{supueps}\sup_{\widetilde Q_{r/2}}|u_\eps|\le \left(\frac{C}{r^{n+3}}\int_{\widetilde Q_r}|u_\eps|^pdxdtd\lambda\right)^{\frac1p}\le \frac{C(\text{diam}(\Omega_T))}{r^{\frac{n+3}{p}}}\|f\|_{L^p(\partial_L \Omega_T)}.\end{equation} 
Thus $u_\eps$ is uniformly bounded with respect to $\eps$ in $L^2(K)$ for any compact subset $K\subset \Omega_T$. By Caccioppoli's inequality, 
$\|\nabla u_\eps\|_{L^2(K)}$ is uniformly bounded in $\eps$. 
Let 
\[B_R(X_0) = \{X\in\R^{n+1}:|X-X_0|<r\}\] 
and let $H^1(B_r(X_0))$ be the Sobolev space defined through the norm 
\[\|v\|_{H^1(B_R(X_0))} = \|v\|_{L^2(B_R(X_0))} + \|\nabla v\|_{L^2(B_R(X_0))}, 
\]
and let $(H^1(B_R(X_0)))^*$ be its dual space. Choose $X_0$ and $t_1<t_2$ such that 
$B_R(X_0)\times (t_1,t_0)$ is compactly contained in $\Omega_T$. From the equation $$\partial_t u_\eps + \LL_\eps u_\eps=0,$$ we see that $\partial_t u_\eps$ is uniformly bounded in $L^2((t_0,t_1);(H^1(B_R(X_0)))^*)$. 

It follows from standard results in  homogenization theory (see \cite[Ch. 11]{CD}) that $\{u_\eps\}_{\eps>0}$ has a subsequence that converges weakly with respect to the norm 
\[
\|u\|_{\mathcal{W}(B_R(X_0)\times(t_1,t_2))} := \|u\|_{L^2(B_R(X_0)\times(t_1,t_2))} + \|\nabla u\|_{L^2(B_R(X_0)\times(t_1,t_2))} + \left\| \partial_t u \right\|_{L^2((t_0,t_1);(H^1(B_R(X_0)))^*)},
\]
to a function $\bar u$ which satisfies $\partial_t \bar u + \bar\LL\bar u=0$ in $B_R(X_0)\times(t_1,t_2)$. 

We shall also need to extract a convergent subsequence of the Kernel $K_\eps$. 
If \begin{equation}\label{xtlambda}(x,t,\lambda)\in B_R(X_0)\times(t_1,t_2)\text{ and }\text{dist}(B_R(X_0)\times(t_1,t_2),\partial\Omega_T)\ge 2r,\end{equation} we get as in \eqref{supueps}, 
\[
\left|\int_{\partial_L\Omega_T} K_\eps(x,t,\lambda;Y,s)f(Y,s)d\sigma(Y)ds\right| = |u_\eps(x,t,\lambda)|\le \frac{C(\text{diam}(\Omega_T))}{r^{\frac{n+3}{p}}}\|f\|_{L^p(\partial_L \Omega_T)}. 
\]
It thus follows by duality that $\|K_\eps(x,t,\lambda;\cdot,\cdot)\|_{L^q(\partial_L\Omega_T)}$ is bounded uniformly in $\eps$ for $(x,t,\lambda)$ as in \eqref{xtlambda}, where $q$ is the conjugate exponent of $p$. This clearly implies that \[\|K_\eps\|_{L^q(B_R(X_0)\times(t_1,t_2)\times\partial_L\Omega_T)}\] is bounded uniformly in $\eps$. Thus, for a subsequence,
$$ K_\eps \longrightarrow \bar K, \quad \text{as } \eps \to 0, \ \text{weakly in } L^q(B_R(X_0)\times(t_1,t_2)\times\partial_L\Omega_T).$$ 
\quad

Suppose $\{u_{\eps_1}\}$ converges weakly in $\mathcal{W}(B_R(X_0)\times(t_1,t_2))$ to $\bar u$. Then, there is a subsequence $\{\eps_2\}$ of $\{\eps_1\}$ 
such that $K_\eps$ converges weakly to $\bar K$ in $L^2(B_R(X_0)\times(t_1,t_2)\times\partial_L\Omega_T)$, as $\eps\to0$, along $\{\eps_2\}$. This yields  
$$\bar u(x,t,\lambda)
= \int_{\partial_L\Omega_T}\bar K(x,t,\lambda;Y,s)f(Y,s)d\sigma(Y)ds.$$ 
Since this holds for any set of the type $B_R(X_0)\times(t_1,t_2)$ that is compactly contained in $\Omega_T$, we conclude that for a certain subsequence of $\{\eps\}_{\eps>0}$, 
$$u_\eps \longrightarrow \bar u, \quad \text{weakly in } \mathcal{W}_{\text{loc}}(\Omega_T),$$   
and 
$$K_\eps \longrightarrow \bar K , \quad \text{weakly in } L^q_{\text{loc}}(\Omega_T)\times L^q(\partial_L\Omega_T),$$ 
where 
\begin{equation*}
\left\{\begin{aligned}
&\partial_t \bar u  + \bar\LL\bar u=0\quad\text{in }\Omega_T,\\
& \int_{\partial_L\Omega_T} \bar K(x,t,\lambda;Y,s)f(Y,s)d\sigma(Y)ds.
\end{aligned}\right.
\end{equation*}
\quad\\
                                 
It remains to prove that $\bar K$ is indeed the kernel associated to $\partial_t + \bar \LL$. That is, we need to show that $\bar u=f$ n.t. on $\partial_L\Omega_T$. Assume that $f$ is smooth. Then by the De Giorgi-Moser-Nash estimate \eqref{eq:MoserBoundary}, $u_\eps$ is uniformly continuous up to the boundary, with estimates uniform in $\eps$. Thus, $u_\eps$ converges uniformly to $\bar u$ in any neighborhood $\mathcal{N}$ of the boundary, for a subsequence, and $\bar u=f$ on $\mathcal{N}\cap\partial D$. Since $\partial_t \bar u + \bar{\LL}\bar u=0$ in $\Omega_T$ we see that 
$$\bar u(x,t,\lambda) 
= \int_{\partial_L\Omega_T} \bar K(x,t,\lambda;Y,s)f(Y,s) d\sigma(Y)ds$$ 
solves the Dirichlet problem \eqref{homogeneous} when $f$ is smooth. Since smooth functions are dense in $L^2$, this proves that $\bar K$ is the kernel associated to $\partial_t+\bar\LL$. 

Finally, taking into account that all convergent subsequences have the same unique limit $\bar u$, we conclude that $u_\eps$ converges locally uniformly, and locally weakly in $\mathcal{W}(\Omega_T),$ to the solution $\bar u$ of \eqref{homogeneous}. 

\section*{Acknowledgments}

\emph{The authors would like to thank Professor Kaj Nystr\"om at Uppsala University for suggesting us to work on this problem. We are grateful to him for the many helpful discussions we had during the course of writing this paper.}


\end{document}